\documentclass[12pt]{amsart}
\usepackage[normalem]{ulem}
\usepackage[utf8]{inputenc}
\usepackage{fullpage, amsmath, amssymb, amsthm, enumerate, url, bibentry, hyperref, color,xcolor}
\usepackage[all,cmtip]{xy}
\usepackage{tikz,epsfig}
\usepackage{verbatim}
\usepackage{geometry}
\usepackage{colonequals}
\usepackage{tabularray}
\usepackage{cleveref}
\usepackage[singlelinecheck=false]{caption}
\usepackage{tikz-cd}

\definecolor{purple}{rgb}{0.59, 0.44, 0.84}

\definecolor{orange}{rgb}{1, 0.6, 0.4}

\newtheorem{theorem}{Theorem}[section]
\newtheorem{lemma}[theorem]{Lemma}
\newtheorem{corollary}[theorem]{Corollary}
\newtheorem{proposition}[theorem]{Proposition}

\newtheorem{defn}{Definition}[section]
\newtheorem{cor}{Corollary}[section]

\newcounter{example}

\newcounter{remark}
\newenvironment{remark}[1][]{\refstepcounter{remark}\par\medskip
   \noindent \textbf{Remark~\theremark. #1} \rmfamily}{\medskip}

\numberwithin{equation}{section}
\numberwithin{example}{section}
\numberwithin{remark}{section}

\newcommand{\ba}{\backslash}

\def \l {\lambda}

\def\P{\mathbb P}
\def\C{\mathbb C}
\def\G{\mathbf G}

\def\wp{\mathfrak{p}}

\def\Tr{{\rm Tr}}
\def\Gal{{\rm Gal}}
\def\Frob{{\rm Fr}}
\def\F{{\mathbb F}}
\def\Z{{\mathbb Z}}

\def\BK{\mathbb{K}}
\def\G{\Gamma}

\def\N{\mathbb N}

\def\CC#1#2{\binom {#1}{#2}}

\def\Q{\mathbb{Q}}

\def \ol{\overline}
\def \eps{\varepsilon}

\newcommand{\fp}{\mathbb{F}_p}

\def \g {\mathfrak{g}}

\newcommand{\fq}{\mathbb{F}_q}

\newcommand*\HYPERskip{&}
\catcode`,\active
\newcommand*\pFq{
\begingroup
\catcode`\,\active
\def ,{\HYPERskip}%
\doHyper
}
\catcode`\,12
\def\doHyper#1#2#3#4#5{%
\, _{#1}F_{#2}\left[\begin{matrix}#3 \smallskip \\  #4\end{matrix} \; ; \; #5\right]%
\endgroup
}

\newenvironment{singnumalign}{
    \begin{equation}
    \begin{aligned}
}{
    \end{aligned}
    \end{equation}
    \ignorespacesafterend
}

\newcommand*\HYPERp{&}
\catcode`,\active
\newcommand*\pPq{
	\begingroup
	\catcode`\,\active
	\def ,{\HYPERp}%
	\doHyperP
}
\catcode`\,12
\def\doHyperP#1#2#3#4#5{%
	\, _{#1}{P}_{#2}\left[\begin{matrix}#3 \smallskip \\  #4\end{matrix} \; ; \; #5\right]%
	\endgroup
}

\crefformat{section}{\S#2#1#3}
\crefformat{subsection}{\S#2#1#3}

\catcode`,\active

\catcode`\,12

\newcommand*\HYPERpp{&}
\catcode`,\active
\newcommand*\pPPq{
\begingroup
\catcode`\,\active
\def ,{\HYPERpp}%
\doHyperFpp
}
\catcode`\,12
\def\doHyperFpp#1#2#3#4#5{%
\, _{#1}{\mathbb P}_{#2}\left[\begin{matrix}#3 \smallskip \\  #4\end{matrix} \; ; \; #5\right]%
\endgroup
}

\def\N{\mathbb N}

\def \HD{\mathit{HD}}
\def \ba{\boldsymbol{\alpha}}
\def \bbeta{\boldsymbol{\beta}}

\begin{document}

\title{The explicit hypergeometric-modularity   method II}

\author{Michael Allen, Brian Grove, Ling Long, Fang-Ting Tu}
\begin{abstract}
    In the first paper of this sequence \cite{HMM1} we provided an explicit hypergeometric modularity method by combining different techniques from the classical, $p$-adic, and finite field settings. In this article, we explore an application of this method from a motivic viewpoint through some known hypergeometric well--poised formulae of Whipple and McCarthy. We first use the method to derive a class of special weight three modular forms, labeled as $\BK_2$-functions. Then using well-poised hypergeometric formulae we further construct a class of degree four Galois representations of the absolute Galois groups of the corresponding cyclotomic fields. These representations are then shown to be extendable to $G_\Q$ and the $L$-function of each extension coincides with the $L$-function of an automorphic form. 
\end{abstract}
\maketitle
\tableofcontents

\section{Introduction} \label{sec:intro}
Properties of  $L$-series from automorphic forms or Galois representations have been one of the central interests in number theory. Geometric $L$-functions are expected to be automorphic which implies their analytic continuation to the whole complex plane.  In particular, there is a class of finite dimensional Galois representations arising from a hypergeometric background that can be studied very explicitly and effectively. These $L$-functions can be viewed as $L$-functions of exponential sums. In this sequence of papers, we explore arithmetic objects with both modular forms and hypergeometric origins. 
In \cite{RRV22}, Roberts and Rodriguez Villegas gave a general description of hypergeometric motives which are governed by hypergeometric data and their associated combinatorics, which was used in papers including \cite{DPVZ,Long-Yang} and others. These hypergeometric motives give a uniform perspective on three 
mutually enriching aspects: the classical complex aspect \cite{AAR}, the $p$-adic aspect initiated by Dwork \cite{Dwork}, and the finite field or Galois aspect pioneered by Greene \cite{Greene} and Katz \cite{KatzESDE}.
\begin{center}
    \includegraphics[height=3.5cm]{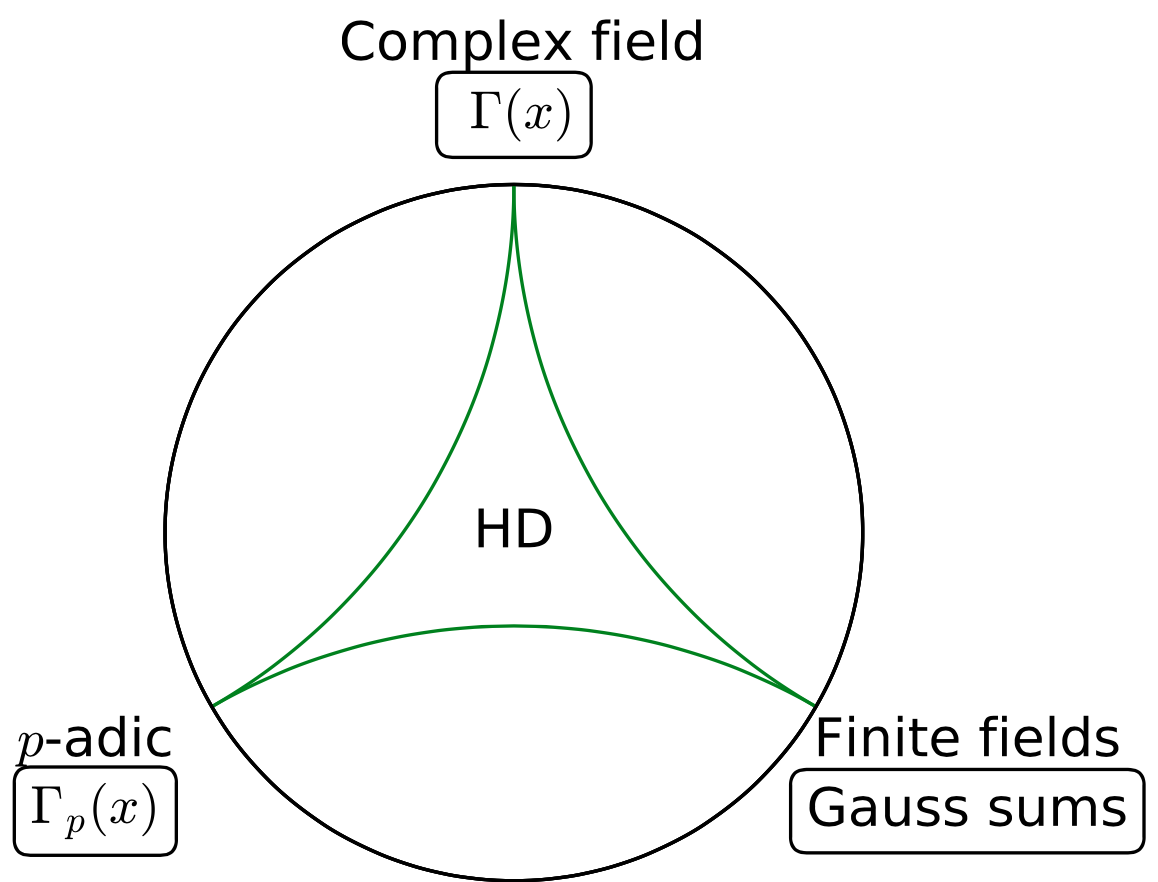}
\end{center}
In this current series, we first proposed the ``Explicit Hypergeometric-Modularity Method" (EHMM) in \cite{HMM1} which provides a direct link between hypergeometric character sums and coefficients of Hecke eigenforms. The method is based on the Euler integral formula for classical hypergeometric functions, Ramanujan's theory of elliptic functions to alternative bases, Dwork's unit root theory, and the Gross--Koblitz formula connecting the finite field and $p$-adic settings.

Here in Part II of this series, we first use the EHMM to construct a class of special weight three cusp forms, called the $\BK_2$-functions. Then we use them and well-poised hypergeometric formulae due to Whipple over $\C$ and McCarthy over finite fields to prove the main results.

We will recall some basic notation. A \emph{hypergeometric} datum $\HD=\left\{\ba, \bbeta\right\}$, or written as $\left\{\substack{\ba\\ \bbeta}\right\}$, where $\ba = \left\{r_1, r_2, \hdots, r_n\right\}$ and $\bbeta = \left\{q_1, q_2, \hdots, q_n\right\}$ are two multisets of rational numbers. Here we assume $q_1=1$.
Let $F(\ba,\bbeta;\l)=\pFq{n}{n-1}{r_1& r_2& \hdots& r_n}{&q_2&\cdots& q_n}\l$. When $q_{1}=1$ for $_nF_{n-1}$, it will be omitted following the classical notation.  As a function of $\l$, $F(\ba,\bbeta;\l)$ is annihilated by an order-$n$ ordinary differential equation. Solutions of this equation (referred to as \emph{periods}) form a rank-$n$ local system.  Combinations of the solutions and  derivatives are referred to as \emph{quasi-periods}.

In \cite{HMM1}, we focus on length three or four data at $\l=1$. Here, based on the EHMM method, we consider $\lambda=-1$ and length four hypergeometric data as follows: for $1\le j\le 11$,  we let 
\begin{equation}\label{eq:D-M-defn}
    D=24/\gcd(j,24),\quad M=\mathrm{lcm}(4,D).
\end{equation} 
The integer $M$ is chosen so that our theorems can be stated cleanly.  Let
\begin{equation}\label{eq:HD4-defn}
    \HD_4\left(\frac{j}{12}\right):=\left\{\begin{array}{cccc}\frac{j}{12}&\frac{j}{12}&\frac12&\frac12\\ 1&1&\frac12 + \frac{j}{12}&\frac12+\frac{j}{12}\end{array}\right\}
\end{equation}
\begin{equation}\label{eq:HD5-defn}
    \HD_5\left(\frac{j}{12}\right):=\left\{\begin{array}{ccccc}\frac{j}{12}&1+\frac{j}{24}&\frac{j}{12}&\frac12&\frac12\\ 1&\frac{j}{24}&1&\frac12 + \frac{j}{12}&\frac12+\frac{j}{12}\end{array}\right\}.
\end{equation}
{These data are \emph{well-poised} which means  $r_i+q_i$ is independent of $i$ and the second one,  obtained by inserting a second column with $r_2=1+r_1/2$ to $\HD_4(r)$, is called \emph{very well-poised}. Very well-poised series has its own interest, see for example \cite{Zudilin-well-poised-Euler} by Zudilin.} The first datum is \emph{primitive}, namely $r_i-q_j\notin \Z$ for each $i,j$; while the second one is not. Whipple in \cite{Whipple25} studied the consequential symmetries at $\l=\pm 1$.  Here we combine the EHMM and Whipple's results to study the splitting of the local system associated with $\HD_4(j/12)$ at $\l=-1$.  The splitting behavior leads us to two families of Hecke modular forms. For brevity,  we describe one weight three family using the Dedekind eta function 
\[
    \eta(\tau)=\eta(q)=q^{1/24}\prod_{n\ge 1} (1-q^n), \text{ where } q=e^{2\pi i \tau}, \tau \in \C, \textnormal{Im} (\tau)>0,
\]
\begin{equation}\label{eq:E}
    \BK_2(r,s)(\tau):= 
    \frac{\eta \left(\frac{\tau}{2} \right)^{16s-8r-12}\eta(2 \tau)^{8s+8r-12}}{\eta(\tau)^{24s-30}}.
\end{equation}
For a modular form $f$, let $L(f,k)$ denote  the $L$-value of $f$ at $k$ (cf. \eqref{eq:L(k)k=1,2} ).
\begin{theorem}\label{thm:mainclassical} For $1\le j\le 11$,
    let $(r_1,q_1)=(1-\frac{j}{24}, 1+\frac{j}{24})$ and $ (r_2,q_2)=( \frac{12-j}{24}, \frac{12+j}{24}).$ Then there exist a holomorphic congruence modular form $K(\frac{j}{12})(\tau)$ and a non-holomorphic one $E(\frac{j}{12})(\tau)$ (see \eqref{eq:K&E}) both of weight two such that
    \begin{equation*} 
    \begin{split}
        2^{2-j/2}\Omega_{j,\C} \cdot F\left(\HD_4\left(j/{12}\right);-1\right)&= \,\, {L}\left(K\left(j/12\right)(\tau),1\right)\cdot { L}\left(\BK_2\left(r_1,q_1\right)(\tau),1\right),\\
        2^{2-j/2}\Omega_{j,\C} \cdot F \left(\HD_5\left(j/{12}\right);-1\right) &= -{L}\left(E\left(j/{12}\right)(\tau),1\right)\cdot { L}\left(\BK_2\left(r_2,q_2\right)(\tau),1\right),
    \end{split}
    \end{equation*}
    where $\displaystyle \Omega_{j,\C}=\frac{\pi }{\sin(j\pi/12)}\frac{\G(j/12)^2}{\G(j/12+1/2)^2}$ and  $\G(\cdot)$ is the usual $\G$ function. We also have the following Legendre relation  (\cite[\S1.6]{BB} or \cite{Silverman2, Waldschmidt}): $$L\left(K\left({j}/{12}\right)(\tau),1\right)\cdot L\left(E\left({j}/{12}\right)(\tau),1\right)=-\frac{3\cdot 2^{5-\frac j2}}{j\sin(\pi j/12)} \pi.$$ 
\end{theorem}
The functions  $\BK_2(r_i,q_i)$ and  $K(\frac{j}{12})(\tau)$ above have multiplicative coefficients.
\begin{theorem}\label{thm:L-Values}
    A suitable  combination of $\BK_2(r_i,q_i)$  (resp. $K(\frac{j}{12})(\tau)$) and their  images under the Hecke operators gives rise to a normalized modular form $f_{3,D}^\sharp$ (resp. $f_{2,D}^\sharp$), depending on $D$, which is an eigenform for all Hecke operators $T_p$ with $p \geq 5$. 
\end{theorem}
\begin{table}[ht]
    \centering
    \begin{tblr}{colspec={|Q[c,m]|Q[c,m]|Q[c,m]|Q[c,m]|Q[c,m]|Q[c,m]|Q[c,m]|Q[c,m]},} \hline
        $j$ & $D(j)$ & $M(j)$ &  $f_{2, D}^{\sharp}$&CM & $f_{3,D}^\sharp$&CM \\  \hline
        $1, 5, 7, 11$ & $24$ & $24$ & $f_{1152.2.d.g}$& $\Q(\sqrt{-6})$  & $f_{1152.3.b.i}$ &No\\  \hline
        $2, 10$ & $12$ & $12$ & $f_{288.2.a.a}$& $\Q(\sqrt{-1})$ & $f_{288.3.g.a}, f_{288.3.g.c}$&No \\  \hline
        $3, 9$ & $8$ & $8$ & $f_{128.2.b.a}$& $\Q(\sqrt{-2})$ & $f_{128.3.d.c}$&No \\  \hline
        $4$ & $6$ & $12$ & $f_{36.2.a.a}$& $\Q(\sqrt{-3})$ & $f_{36.3.d.a}, f_{36.3.d.b}$&  $\Q(\sqrt{-1})$ \\  \hline
        $6$ & $4$ & $4$ & $f_{32.2.a.a}$& $\Q(\sqrt{-1})$ & $f_{32.3.c.a}$ &No\\ \hline
        $8$ & $3$ & $12$ &  $f_{36.2.a.a}$& $\Q(\sqrt{-3})$ & $f_{36.3.d.a}, f_{36.3.d.b}$&  $\Q(\sqrt{-1})$ \\  \hline
    \end{tblr}
    \caption{The cuspforms  in LMFDB labels \cite{LMFDB} with CM information}
    \label{tab:main}
\end{table}
In particular, $f_{2,D}^\sharp$ is a CM modular form, see \Cref{tab:f2D}. More information on $f_{2,D}^{\sharp}$ will be given in \Cref{thm:mainpadic} when we consider the $p$-adic aspect. In Dwork's theory, one considers the truncation $F(\ba,\bbeta;\l)_{p-1}$ of  $F(\ba,\bbeta;\l)$ after the $p^{th}$ term.  We use the $p$-adic Gamma functions $\G_p$, discussed further in \Cref{sec:padic} below.  For a normalized Hecke eigenform $f$, we use $a_n(f)$ to denote its $n^{th}$ coefficient and  $\rho_f$ for the corresponding Deligne representation of the absolute Galois group of $\Q$.  In particular, because $f_{2,D}^\sharp$ is CM, $\rho_{f_{2,D}^\sharp}$ decomposes as the sum of two characters when restricted to $G(M)$ for $G(M) \colonequals \mathrm{Gal}(\overline{\Q}/\Q(\zeta_M))$, with $\zeta_M$ a primitive $M^{th}$ root of unity.
\begin{theorem} \label{thm:mainpadic}
    Let $1\le j\le 6$ and $D$ and $M$ be defined as in \eqref{eq:D-M-defn}. For any prime $p\equiv 1\pmod {M}$, there exist $\left(a_p(f_{2,D}^\sharp)\right)_0$ and $\left(a_p(f_{2,D}^\sharp)\right)_1$   in $\overline \Q\cap \overline \Q_p$  of $p$-adic orders 0 and 1, respectively, such that
    \[
        a_p(f_{2,D}^\sharp)=\left (a_p(f_{2,D}^\sharp)\right)_0+\left (a_p(f_{2,D}^\sharp)\right)_1, \quad \text{and } \left (a_p(f_{2,D}^\sharp)\right)_0\cdot\left (a_p(f_{2,D}^\sharp)\right)_1=p.
    \]
    Moreover, if we let $\displaystyle \Omega_{j,\Q_p}=\frac{\G_p(j/12)^2}{\G_p(j/12+1/2)^2,}$ then  
    \begin{align}
        \Omega_{j,\Q_p}F(\HD_4(j/12);-1)_{p-1}&\equiv 
        \left (a_p(f_{2,D}^\sharp)\right)_0 a_p(f_{3,D}^\sharp) \pmod{p^2}; \label{eq:length-4-super} \\
        \Omega_{j,\Q_p}F(\HD_5(j/12);-1)_{p-1}&\equiv
        \left (a_p(f_{2,D}^\sharp)\right)_1a_p(f_{3,D}^\sharp)\pmod{p^2}. \label{eq:length-5-super}
    \end{align}
\end{theorem}
For the $\ell$-adic aspect, we let $H_q(\HD;\l;\wp)$ be the finite field hypergeometric function (see \eqref{eq:P-H}).
\begin{theorem} \label{thm:mainGalois}
    Let $1\le j\le 11$,  $D$ and $M$ be defined as in \eqref{eq:D-M-defn}, and  $f_{2,D}^\sharp$  and $f_{3,D}^\sharp$ be  as in \Cref{thm:L-Values}, where $\rho_{f_{2,D}^\sharp} |_{G(M)}\simeq\chi_{D,1}\oplus \chi_{D,2}.$ For each prime ideal $\wp$ of $\Z[\zeta_M]$ above $p\equiv 1\pmod M$, we have
    \begin{equation}\label{eq:point-counts}
         \Omega_{j,\F_p} H_p\left(\HD_4(j/12);-1;\wp\right)=a_p(f_{2,D}^\sharp)a_p(f_{3,D}^\sharp),  
    \end{equation} 
    where $\displaystyle \Omega_{j,\F_p}=p\frac{\g(\iota_\wp(j/12))^2}{\g(\iota_\wp(j/12+1/2))^2}$, $\g(\cdot)$ denotes Gauss sum, and $\iota_\wp$ is given by \eqref{eqn:residue symbol}.

   Equivalently, if we let $\rho_{\{\HD_4(j/12);-1\}}$ be an explicit hypergeometric representation of $G(M)$ (see \Cref{thm:Katz}) then $\rho_{\{\HD_4(j/12);-1\}}$ is automorphic. In particular,
    \[
        \rho_{\{\HD_4(j/12);-1\}}\simeq \left(\rho_{f_{2,D}^\sharp}\otimes \rho_{f_{3,D}^\sharp}\right)|_{G(M)}=\left[\left(\chi_{D,1}\oplus\chi_{D,2}\right)\otimes \rho_{f_{3,D}^\sharp} \right]|_{G(M)}.
    \]
\end{theorem}
\Cref{fig:1} illustrates the underlying geometry.  Namely, this decomposition arises from the action of an involution $\mathfrak{i}$ \eqref{eq:iota-map} on the representation space of $\rho_{\{\HD_4(j/12);-1\}}$. The existence of this involution is due to the fact that $\HD_4(j/12)$ is well-poised.
\begin{figure}[ht]
    \centering
    \includegraphics[scale=0.75]{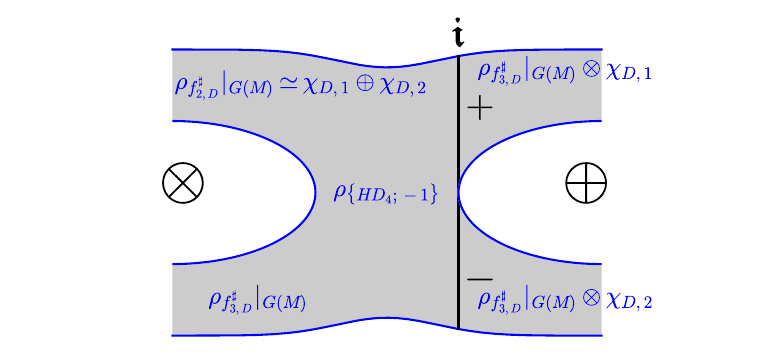}
    \captionsetup{justification=centering}
    \caption{Decomposition of the representation.}
    \label{fig:1}
\end{figure}

The choices of $f_{2,D}^\sharp$  and $f_{3,D}^\sharp$ , or the choices of $K(j/12)$  and $\BK_2(r_i,q_i)$  in \Cref{thm:mainclassical} are not unique,  see the appendix for more discussion.  

One can realize the corresponding hypergeometric representation from the third \'etale cohomology group of the   smooth model of the following three-fold
\begin{equation} \label{eq:hgv}
y^{12} = x_{1}^6 x_2^6 x_{3}^{12-j}(1-x_1)^{12-j}  (1-x_2)^{12-j}(1-x_3)^j(1+ x_{1} x_2x_3)^{j},
\end{equation} when $\gcd(j,12)=1$, otherwise adjusting the exponents by dividing $D$ will give the corresponding model. The well-poised assumption on $\HD_4(j/12)$ at $\l=-1$ can be realized through the following involution acting on \eqref{eq:hgv}:
\begin{equation}\label{eq:iota-map}
        {\mathfrak {i}}:\begin{cases}
            x_i&\mapsto 1/x_i, \quad  i=1,2,3\\
            y&\mapsto (-1)^{j/12} y/(x_1x_2x_3)^2. 
        \end{cases}
    \end{equation}
Equivalently, over finite fields of cardinality $q \equiv 1 \pmod{M}$ the values in \eqref{eq:point-counts} correspond to point counts of this variety by \cite[Proposition 4.2]{Win3X}.

Note the $j=6$ case in \Cref{thm:mainGalois} was obtained  in \cite{McCarthy-Papanikolas} by  McCarthy--Papanikolas.  Below we summarize a few conclusions for this case.
\begin{proposition}\label{prop:otheraspects}
    When $j=6$, $\HD_4(1/2)=\left\{\substack{\frac12,\frac12,\frac12,\frac12\\ 1,1,1,1}\right\}$. $f_{3,6}^\sharp=f_{32.3.c.a}$ and $f_{2,6}^\sharp=f_{32.2.a.a}$.
    \begin{itemize}   
        \item Over $\C$:
        \begin{equation}\label{eq:untruncated_version}
            \begin{split}
                \pFq{4}{3}{\frac{1}{2}&\frac{1}{2}&\frac{1}{2}&\frac{1}{2}}{&1&1&1}{-1} &= \frac{32}{\pi^{2}}L(f_{32.2.a.a},1)\cdot \mid \mathrm{Im}(L(f_{32.3.c.a},1))\mid,\\
                \pFq{5}{4}{\frac{1}{2}&\frac54&\frac{1}{2}&\frac{1}{2}&\frac{1}{2}}{&\frac14&1&1&1}{-1} &= \frac{2}{\pi L(f_{32.2.a.a},1)}\cdot  \mathrm{Re}(L(f_{32.3.c.a},1)), 
            \end{split}
        \end{equation}
        where $L(f_{32.2.a.a},1) = \frac 18 \frac{\Gamma\left(\frac{1}{4}\right)\Gamma\left(\frac12\right)}{\Gamma\left(\frac34\right)};$ and
        \begin{multline}\label{eq:L(32.3.c.a,1}
            L(f_{32.3.c.a},1)= \frac{\G(\frac12)}8\left(\frac{\G(\frac14)}{\G(\frac34)}\pFq32{\frac12&\frac12&{ \frac14}}{&1&{\frac34}}{1} +4\sqrt{-1} \frac{\G(\frac34)}{\G(\frac14)} \pFq32{\frac12&\frac12&{\frac34}}{&1&{\frac54}}{1}\right).
        \end{multline}
        \item Over $\F_p$: For each prime $p>3$,
        \begin{equation}\label{eq:FF_version}
            H_{p}\left[\begin{matrix} \frac{1}{2}&\frac1{2}&\frac{1}{2}&\frac{1}{2} \smallskip \\  1&1&1&1 \end{matrix} \; ; \; -1 \right] =a_p(f_{32.2.a.a}) a_{p}(f_{32.3.c.a}).
        \end{equation}
        In this case, the L-function  $L(\HD_4(1/2);-1;s)$
        \footnote{It can be computed a \texttt{Magma} package implemented by Watkins \cite{Watkins-HGM-documentation} (cf. also \cite{Cohen-Compute-L} by H. Cohen).} is the Rankin--Selberg convolution of $L(f_{32.2.a.a},s)$ and $L(f_{32.3.c.a},s)$.
        \item The $p$-adic version: For each prime $p\equiv 1\pmod 4,$ 
        \begin{equation}
            a_p(f_{32.2.a.a})=\left(a_p(f_{32.2.a.a})\right)_0+p/\left(a_p(f_{32.2.a.a})\right)_0,\, \text{and } \left(a_p(f_{32.2.a.a})\right)_0= \frac{\G_p(\frac14)\G_p(\frac12)}{\G_p(\frac34)}.
        \end{equation} 
        Then
        \begin{singnumalign}\label{eq:truncated_version}
            \pFq{4}{3}{\frac{1}{2}&\frac{1}{2}&\frac{1}{2}&\frac{1}{2}}{&1&1&1}{-1}_{p-1} &\equiv p \cdot \pFq{3}{2}{\frac{1}{2}&\frac{1}{2}&\frac{3}{4}}{&1&\frac{5}{4}}{1}_{p-1} \\ 
            &\equiv \left(a_p(f_{32.2.a.a})\right)_0 a_p(f_{32.3.c.a})\pmod{p^2}.  
        \end{singnumalign}
        \begin{singnumalign}\label{eq:truncated_version2}
            \pFq{5}{4}{\frac{1}{2}&\frac54&\frac{1}{2}&\frac{1}{2}&\frac{1}{2}}{&\frac14&1&1&1}{-1}_{p-1}  &\equiv p \cdot \pFq{3}{2}{\frac{1}{2}&\frac{1}{2}&\frac{1}{4}}{&1&\frac{3}{4}}{1}_{p-1}\\
            &\equiv \frac{p}{\left(a_p(f_{32.2.a.a})\right)_0}a_p(f_{32.3.c.a}) \pmod{p^2}.   
        \end{singnumalign}
    \end{itemize}
\end{proposition}
This paper is organized as follows. We first provide some background in \Cref{s:prelim}, then  in \Cref{sec:complexAspect} we recall the EHMM method and use it to construct the $\BK_2$-functions and derive their properties. Well-poised formulae will be recalled and used in \Cref{sec:4dimGalois}. Then we prove our main results over $\C$ and $\F_q$. The $p$-adic aspect is discussed in \Cref{sec:padic}. In the appendix, we use the hypergeometric values to explore special  $L$-values of the $\BK_2$-functions and some of their quadratic twists.  
\subsection{Acknowledgements}The authors would like to express their sincere gratitude to Frits Beukers, Henri Darmon, Wen-Ching Winnie Li, Tong Liu, Dermot McCarthy, Ken Ono, Esme Rosen, Hasan Saad, Armin Straub, John Voight, Pengcheng Zhang, and Wadim Zudilin for their discussions and suggestions. Special thanks to Esme Rosen who computed the forms in \Cref{tab:main} and \Cref{tab:f2D} and shared with us many symmetries of the $\BK_1$ and $\BK_2$ functions.

Grove was partially supported by a summer research assistantship from the Department of Mathematics at Louisiana State University (LSU). Long is supported in part by the Simons Foundation grant  \#MP-TSM-00002492 and the LSU Michael F. and Roberta Nesbit McDonald Professorship; Tu is supported by the NSF grant DMS \#2302531.
\section{Notation and Hypergeometric Functions}\label{s:prelim}
\subsection{Notation and classical theory}\label{ss:2.1.1}
Here we briefly recall the necessary notation and definitions from Part I of this series.  For a more thorough introduction to hypergeometric functions and character sums in our setting, we refer the reader to \cite{HMM1}.  We use $\HD$ to denote a \emph{hypergeometric datum} $\left\{\ba, \bbeta\right\}$ or $\left\{\substack{\ba\\ \bbeta}\right\}$, where $\ba = \left\{r_1, r_2, \hdots, r_n\right\}$ and $\bbeta = \left\{q_1, q_2, \hdots, q_n\right\}$ are two multisets of rational numbers.  We say that $\HD$ is \emph{primitive} if $r_i - q_j \notin \Z$ for all $1 \leq i, j\leq n$, and that $\HD$ is \emph{defined over $\Q$} if both
\[
    \prod_{j=1}^n x-e^{2\pi i r_j} \quad \text{and} \quad \prod_{j=1}^n x-e^{2\pi i q_j}
\]
have integer coefficients.  We refer to $n$ as the \emph{length} of $\HD$.

One useful aspect of classical hypergeometric functions is the integral representation of Euler. When $\text{Re}(q_{i})>\text{Re}(r_{i})>0$, this gives an inductive formula to construct hypergeometric functions, see for example \cite[(2.2.2)]{AAR}.  In the following, we assume that $q_1 = 1$. First, define $_{1}P_{0}[r_{1};z]\colonequals (1-z)^{-r_{1}}$. Then for $n \geq 2$ define the (classical) period function as
\begin{multline}\label{eq:inductive}
    \pPq{n}{n-1}{\, \, r_{1}&r_{2}&\cdots&r_{n}}{&q_{2}&\cdots&q_{n}&}{z}=\\
    \int_{0}^{1} t^{r_{n}}(1-t)^{q_{n}-r_{n}-1} \pPq{n-1}{n-2}{r_{1}&r_{2}&\cdots&r_{n-1}}{&q_{2}&\cdots&q_{n-1}}{zt}\frac{dt}{t}.
\end{multline} 
The classical $_{n}F_{n-1}$ function is related to $_{n}P_{n-1}$ by $_{1}F_{0}[r_{1};z]:=\, _{1}P_{0}[r_{1};z]$, and
\begin{equation}\label{eq:FtoP}
    \begin{split}
        &\pFq{n}{n-1}{r_{1}&r_{2}&\cdots&r_{n}}{&q_{2}&\cdots&q_{n}}{z}= \prod_{i=2}^{n} \frac{\Gamma(q_{i})}{\Gamma(r_{i})\Gamma(q_{i}-r_{i})} \cdot \pPq{n}{n-1}{r_{1}&r_{2}&\cdots&r_{n}}{&q_{2}&\cdots&q_{n}}{z},
    \end{split} 
\end{equation}
For convenience, in this series of papers, we adopt the convention that 
\begin{equation}\label{eq:shorthand-quotient}
    f\left(\frac{a_{1}^{r_{1}},\cdots, a_{m}^{r_{m}}}{b_{1}^{s_{1}},\cdots,b_{n}^{s_{n}}}\right)\colonequals\frac{f(a_{1})^{r_{1}}\cdots f(a_{m})^{r_{m}}}{f(b_{1})^{s_{1}}\cdots f(b_{n})^{s_{n}}}
\end{equation}
where $f$ could be the Gamma function $\G(\cdot)$, the Pochhammer symbol $(\cdot)_n$, the $p$-adic Gamma function $\G_p(\cdot)$, or the Gauss sum $\g(\cdot)$ when there is no ambiguity. In a like manner, $f\left({a_{1}^{r_{1}},\cdots, a_{m}^{r_{m}}}\right)$ is used to denote $f(a_1)^{r_{1}}\cdots f(a_m)^{r_{m}}$.  Moreover, for a fixed multiset $\ba = \left\{r_1, \hdots, r_n\right\}$,  $f(\ba):=f(r_1, r_2, \hdots, r_n)$.
\subsection{Some background over the finite fields}\label{ss:3F_p}
Let $\F_q$ be a finite field of odd characteristic and $\widehat{\F_q^\times} = \langle \omega \rangle$ be the set of multiplicative characters for $\F_q^\times$. We take the convention that $A(0) = 0$ for each $A\in \widehat{\F_q^\times}$. Let $\eps$ be the trivial character, $\phi$ be the quadratic character, and $\overline A$ be the inverse of $A$. 

Further, the following shorthand notations for Gauss and Jacobi sums are used. Let $a,b \in \Q$ such that $a(q-1),b(q-1)\in\Z$ then
\[
    \g_{\omega}(a) := \sum_{x \in \fq} \omega^{(q-1)a}{(x)}\zeta_{p}^{\Tr_{\fp}^{\mathbb{F}_{q}}(x)}
\]
and
\[
    J_{\omega}(a,b) := \sum_{x \in \fq} \omega^{(q-1)a}{(x)}\omega^{(q-1)b}(1-x),
\]
where $\Tr_{\fp}^{\mathbb{F}_{q}}$ is the standard trace map from $\fq$ to $\fp$. 
\subsubsection{Hypergeometric character sums}\label{ss:2.1.4}
In \cite{Win3X}, a finite character sum $_{n}\mathbb{P}_{n-1}$ is defined inductively, parallel to \eqref{eq:inductive}. We now describe how to pass from $\Z[\zeta_M]$ to  finite fields. For each nonzero prime ideal $\wp$ coprime to $M$ and integer $i$ we assign a character to the finite residue field $\kappa_\wp\colonequals\Z[\zeta_M]/\wp$ using the $M^{th}$ residue symbol by 
\begin{equation}\label{eqn:residue symbol}
    \iota_\wp\left(\frac iM\right)(x) \colonequals \left( \frac x \wp\right)_M^i\equiv x^{(q-1)\frac{i}M} \pmod {\wp}, \quad { \forall x\in \Z[\zeta_M]. }
\end{equation} 
For example, if $\wp$ is coprime to $2$, $\iota_\wp(1/2)=\phi_\wp$ or simply $\phi$, the quadratic character. Likewise, $\iota_\wp(1)$ is the trivial character, denoted by $\eps_\wp$ or $\eps$. By definition $\iota_\wp(i/M)(0)=0$ for each $i$. In this notation,  we  denote Gauss sums by $\g(\iota_\wp(r_i))$ and Jacobi sums by $J(\iota_\wp(r_i),\iota_\wp(r_j)).$ For simplicity, we write $R_i$ for $\iota_\wp(r_i)$, $\overline{R}_i$ for $\iota_\wp(-r_i)$, and define $Q_i$ and $\overline{Q}_i$ analogously and use $q$ for $|\kappa_\wp|$. We refer the reader to \cite{Berndt-Evans-Williams} or \cite{Win3X} for further background on Gauss and Jacobi sums.  We now replace our classical hypergeometric functions with character sums using the dictionary first used by Greene \cite{Greene} and described explicitly in \cite[\S2.4]{Win3X}. The first key adjustment is to replace $t^{r_{i}}$ with $R_{i}(t)$, where $t \in \fq$. Now the Pochhammer symbols in the classical setting are quotients of gamma functions. The gamma functions are then replaced with Gauss sums. Let $R_i,Q_i,q$ as above.  For a fixed $\l \in \F_q^\times$ and $n \geq 1$, define 
\begin{multline}\label{eq:PP}
    \P(\ba,\bbeta;\l;\wp)=  \pPPq{n}{n-1}{R_1& R_2&\cdots &R_{n}}{ \, Q_1&  Q_2&\cdots &Q_{n}}{\l;q}\\
    \colonequals \frac{(-1)^{n}}{q-1}\cdot \left(\prod_{i=2}^{n} R_iQ_i(-1) \right)\cdot \sum_{\chi\in \widehat{\F_q^\times}}\CC{R_1\chi}{ Q_1\chi} \CC{R_2\chi}{Q_2\chi}\cdots \CC{R_{n}\chi}{Q_{n}\chi}\chi(\l),
\end{multline}
where $\displaystyle\CC AB \colonequals-B(-1)J(A,\ol B) = -B(-1) \sum_{x \in \fq} A(x) \overline B(1-x),$  taking values in $\Q(\zeta_M)$ or $\Q_\ell(\zeta_M)$.  When the  pair $(\ba,\bbeta)$ is primitive, let
\begin{equation}\label{eq:P-H}
    H_q(\ba,\bbeta;\l;\wp)  :  ={(-1)^{n-1}}\mathcal J(\HD;\wp)^{-1} \P(\ba,\bbeta;\l;\wp), 
\end{equation} 
where
\begin{equation}\label{eq:calJ}
    \mathcal J(\HD;\wp)\colonequals  \iota_\wp(r_1+q_1)(-1)\prod_{i=1}^n -J(\iota_\wp(r_i),\iota_\wp(q_i-r_i)). 
\end{equation}
Note that $H_q(\ba,\bbeta;\l;\wp)$ can be written in terms of Gauss sums as in \cite{BCM,McCarthy} or (4.3) of \cite{HMM1}. If $H_q(\ba,\bbeta;\l;\wp)\in \Z$, independent of which prime ideal $\wp$ above $(p)=\wp\cap \Z$, we replace $\wp$ by $p$ as in \cite{HMM1}.  In certain applications, especially in \S \ref{sec:padic}, we will fix an embedding of $\Q(\zeta_M)$ to $\overline \C_p$   by taking  $\iota_\wp\left(\frac 1{q-1}\right) \colonequals \omega_q$ to be the Teichm\"uller character $\omega_q$. Under this embedding, we write $H_q(\ba,\bbeta;\l;\wp)$ as $H_q(\ba,\bbeta;\l;\omega_q)$.
\subsection{Hypergeometric Galois representations} 
To motivate the connection between character sums and Galois representations, we call a result of Weil which allows us to view a Jacobi sum as a Gr$\ddot{\text{o}}$ssencharacter. Below $G(M)$ denotes the absolute Galois group of $\Q(\zeta_M)$ as before.
\begin{theorem}[Weil \cite{Weil52}]\label{thm:Weil} 
    Let $r,s\in \Q$ such that $r,s,r+s\notin \Z$. Let $M= \mathrm{lcd}(r,s)$.  Then there exists a representation $\chi$ of $G(M)$ such that at each nonzero prime ideal $\wp$ of $\Z[\zeta_M]$ coprime to $M$,
    \[
        \chi(\text{Frob}_\wp)=-J(\iota_\wp(r),\iota_\wp(s)),
    \]
    in which $\text{Frob}_\mathfrak{p}$ denotes the geometric Frobenius conjugacy class of $G(M)$ at $\mathfrak{p}$. 
\end{theorem}
An important result that bridges the finite field and $p$-adic settings is the Gross-Koblitz formula, see \cite[\S 2.4]{HMM1}. Here is a version of this formula for Jacobi sums.
\begin{lemma}[Corollary 2.7 of \cite{HMM1}]\label{cor:Jac-Gp}
    For $r,s\in \Q\cap (0,1)$, $p\equiv 1\pmod{\mathrm{lcd}(r,s)}$ such that $r+s< 1$.  Then 
    \begin{equation}\label{eq:5.31}
        -J_{\bar \omega_p}(r,s)=
        \Gamma_{p}\bigg(\frac{r,s}{r+s} \bigg).
    \end{equation}
\end{lemma} 
Next we recall the following result of Katz which puts hypergeometric character sums into the Galois perspective. 
\begin{theorem}[Katz \cite{KatzESDE, Katz09}]\label{thm:Katz}  
    Let $\ell$ be a prime. Given a primitive pair of multi-sets $\ba=\{r_1,\cdots,r_n\}$, $\bbeta=\{q_1=1,q_2,\cdots,q_n\}$ with $M = M(\HD)$, for any $\l \in \Z[\zeta_M,1/M]\smallsetminus \{0\}$ the following hold: 
    \begin{itemize}
        \item [i).]There exists an $\ell$-adic Galois representation $\rho_{\{\HD;\l\}}: G(M)\rightarrow GL(W_{\l})$ unramified almost everywhere such that at each prime ideal $\mathfrak{p}$ of  $ \Z[\zeta_M,1/(M\ell \l)]$ with residue field $\kappa_\mathfrak{p}$ and norm $q=N(\mathfrak{p})=|\kappa_\mathfrak{p}|$,
        \begin{equation}\label{eq:Tr1} 
            \Tr \rho_{\{\HD;\l\}}(\text{Frob}_\mathfrak{p})= (-1)^{n-1}  \iota_{\mathfrak{p}}(r_1) (-1)\cdot \P(\ba,\bbeta; 1/\l;\mathfrak{p}).  
        \end{equation}
        \item[ii).] When $\l\neq 1$,  the dimension $d \colonequals dim_{\overline \Q_\ell}W_{\l}$ equals $n$ and all roots of the characteristic polynomial of $\rho_{\{\HD;\l\}}(\Frob_\mathfrak{p})$  are algebraic numbers and have the same absolute value $N(\mathfrak{p})^{(n-1)/2}$ under all archimedean embeddings. If $\rho_{\{\HD;\l\}}$ is self-dual, namely isomorphic to any of its complex conjugates, then $W_{\l}$ admits a non-degenerate alternating (resp. symmetric) bilinear pairing if $n$ is even and $\sum_{i=1}^n (q_i-r_i)\in\Z$ (resp. otherwise).
        \item[iii).] When $\l=1$,  $d=n-1$.  If $\rho_{\{\HD;\l\}}$ is self-dual, it contains a subrepresentation that admits a non-degenerate alternating (resp. symmetric) bilinear pairing if $n$ is even and $\sum_{i=1}^n (q_i-r_i)\in\Z$ (resp. otherwise). For this subrepresentation, the roots of the characteristic polynomial of $\mathrm{Frob}_{\mathfrak{p}}$ have absolute value exactly $N(\mathfrak{p})^{(n-1)/2}$.
    \end{itemize}
\end{theorem}
\begin{remark}
    This result of Katz holds for more general hypergeometric data with $q_{1} \neq 1$, but we use the above statement for simplicity.
\end{remark}
\section{The \texorpdfstring{$\mathbb{K}_2$}{K2}-functions and the Proof of \texorpdfstring{\Cref{thm:L-Values}}{Theorem 1.2}} \label{sec:complexAspect}
\subsection{The explicit hypergeometric-modularity method}
In the first paper of this series \cite{HMM1} we establish various congruences between the Fourier coefficients of newforms, truncated classical hypergeometric functions, and finite field hypergeometric functions to establish the modularity of certain hypergeometric Galois representations. For the reader's convenience, we begin by restating the main result of the first paper.
\begin{theorem}[\cite{HMM1}]\label{thm:main}
    Let $n \in \{3,4\}$. Assume $\ba^\flat=\{r_1,\cdots,r_{n-1}\}$, where $0<r_1
    \le\cdots\le r_{n-1}<1$, and $\bbeta^\flat=\{1,\cdots,1\}$, with $r_n,q_n$ such that  $0<r_n<q_n\le 1$ and $r_2<q_n$.  Let $\HD=\{\{r_n\}\cup \ba^\flat,\{q_n\}\cup \bbeta^\flat\}$ and $M=M(\HD)$. Suppose $\HD$ satisfies the following hypotheses:
    \begin{enumerate}
        \item There exists a modular function $t=C_1 q^{}+O(q^2) \in\Z[[q]]$ such that 
        \begin{equation}\label{eq:f}
            f_{\HD}(q)\colonequals C_1^{-r_n}\cdot t(q)^{r_n}(1-t(q))^{q_n-r_n-1}F(\ba^\flat,\bbeta^\flat;t(q))\frac{dt(q)}{t(q)dq}
        \end{equation}
        is a congruence weight $n$ holomorphic cuspform satisfying that, for each prime $p \in P_M=\{p\mid p\equiv 1\pmod M\}$, $T_p f_{\HD}=\tilde b_p  f_{\HD}$ for some $\tilde b_p$  in $\Z$   where $T_p$ is the $p$th Hecke operator.
        \item For any prime ideal $\wp$ in $\Z[\zeta_M]$ above $p\in P_M$ we have 
        \begin{equation}\label{eq:PinZ}
            \chi_{\HD}(\wp) \cdot \P\left (\HD;1;\wp\right)\in \Z,
        \end{equation} 
        where 
        \begin{equation}\label{eq:chi-HD}
            \chi_{\HD}(\wp)\colonequals\iota_\wp(r_n)(C_1)^{-1}  \cdot \prod_{i=2}^{n-1}\iota_\wp(r_i)(-1).
        \end{equation}
\end{enumerate}
Then there exists a normalized Hecke eigenform $f_{\HD}^\sharp$, not necessarily unique, such that for each $p \in P_M$, $\tilde b_p=a_p(f_{\HD}^\sharp)$, and at each prime $p\ge 29$ in $P_M$ 
\begin{equation}\label{eq:2.9}
    { (-1)^{n-1}} \chi_{\HD}(\wp)\cdot \P(\HD;1;\wp)=a_p(f_{\HD}^\sharp)+\delta_{\gamma(\HD)=1}\cdot \psi_{\HD}(p) \cdot p.
\end{equation} 
Here $\delta_{\gamma(\HD)=1}$ is equal to 1 when ${\gamma(\HD)=1}$ and is 0 otherwise and
\begin{equation}\label{eq:sgn}
   \psi_{\HD}(p) \equiv (-1)^{n-1}\cdot \iota_p(r_n)(C_1)\G_p\left(\frac{ q_n-r_n}{\ba^\flat}\right) \pmod p. 
\end{equation}
\end{theorem}
\subsection{Construction of the \texorpdfstring{$\BK_2$-functions}{K2-functions}}
Throughout this paper, $n=3$, and so we suppress the subscript on $r_n$ and $q_n$.  To avoid confusion between the parameter $q_3$ and $q=e^{2\pi i \tau}$ we use $s$ for $q_3$ in this paper.

In this paper we let $\ba^\flat=\{\frac12,\frac12\},\bbeta^\flat=\{1,1\},$ and $t$ to be the modular lambda function 
\begin{equation}\label{eq:lambda-in-eta}
  \l(\tau)=16 \frac{\eta(\frac{\tau}{2})^{8}\eta(2 \tau)^{16}}{\eta(\tau)^{24}}.
\end{equation}
In the classic theory,  one has the following when both sides make sense
\begin{equation}\label{eqn: alt-2}
   \pFq21{\frac12&\frac12}{&1}{\l(\tau)}=\sum_{n,m\in \Z}q^{(n^2+m^2)/2}=\theta_3(\tau)^2,
\end{equation}
where $\theta_3(\tau)$ is a weight $\frac{1}{2}$ Jacobi-theta function \cite{Zagier-modularform}. Using \eqref{eq:lambda-in-eta} and \eqref{eqn: alt-2}, we have
\begin{lemma}[\cite{HMM1}]\label{lem:E-eta}
    Given $r,s\in \Q$, 
    \begin{equation}\label{eq:alt-2-eta}
        {\l}(\tau)^{r}{(1-\l(\tau))}^{s-r-1}   \pFq21{\frac 12&\frac 12}{&1}{\l(\tau)}  \frac{d \l(\tau)}{\l(\tau)dq} = 2^{4r -1 } \BK_2(r,s)(\tau),
    \end{equation} 
    where as in the introduction
    \begin{equation*}
        \BK_2(r,s)(\tau):= \frac{\eta \left(\frac{\tau}{2} \right)^{16s-8r-12}\eta(2 \tau)^{8s+8r-12}}{\eta(\tau)^{24s-30}}.
    \end{equation*}
\end{lemma}
Further, we define a particular set of parameters by
\begin{equation}\label{eq:S-defn}
   {\mathbb S}_2:=\left\{(r,s)\mid 0<r<s<3/2, r\neq 1, s\neq \frac12, \, 24s\in \Z\, , 8(r+s)\in\Z\right\}.
\end{equation}
The set $ {\mathbb S}_2$ is chosen such that for each $(r,s)$ in $\mathbb S_2$, $\mathbb K_2(r,s)$ is a congruence holomorphic cuspidal modular form of weight three, and the datum $\left\{\substack{\frac12,\frac12,r\\ 1,1,s}\right\}$ is primitive. 
\begin{defn}[\cite{HMM1}]\label{defn:K2}
    Given $r,s\in \Q$, define  
    \begin{equation}\label{eq:E-E*}
        \BK_2^*(r,s)(\tau)\colonequals \BK_2(s-r,s)(\tau);
    \end{equation}
    and for $r\in \Q_{\ge0}$ such that $24r\in \Z$ we define
    \begin{equation}\label{eq:N(r)}
        N(r):=\frac {48}{\gcd(24r,24)}.
    \end{equation}
\end{defn}
\begin{lemma}[Lemma 3.1  \cite{HMM1}]\label{lem:K2-level}
    For each $(r,s)\in  {\mathbb S}_2$, $\BK_2(r,s) (N(r)\tau)$ is a congruence weight three holomorphic cusp form of  level $ N(r)N(s-r)$ with Dirichlet character $\left(\frac{-2^{24s}}\cdot\right)$. 
    Moreover, $\BK_2(r,s)(-1/\tau)=2^{4s-8r}i\cdot \tau^3 \BK^{*}_2(r,s)(\tau)$. 
\end{lemma} 
\subsection{The Galois and non-Galois Cases} \label{ss:K2-Galois}
Back to Theorem \ref{thm:main}, for each $(r,s)\in \mathbb S_2$, assumption (1) is satisfied by the construction of $\mathbb S_2$. The next definition addresses when assumption (2), namely when \eqref{eq:PinZ} holds, for $\HD=\{\{\frac12,\frac12,r\},\{1,1,s\}\}$.
\begin{defn}\label{defn:galois}
    A pair $(r,s)\in\mathbb S_2$  is said to be in a \emph{Galois orbit} for the $\BK_2$-family if for any $c\in (\Z/N\Z)^\times$ where $N=\mathrm{lcd}(r,s)$, there exists $(r_c,s_c)\in \mathbb S_2$ such that $cr-r_c, cs-s_c\in\Z$,     such that $\mathbb K_2(r,s)$ and $\mathbb K_2(r_c,s_c)$ are in the same Hecke orbit. In this case, we use $f_{\HD}^\sharp$ to denote a normalized Hecke eigenform in the same orbit.
\end{defn}

Let $O$ be the Galois orbit associated to $HD$. Then $f_{\HD}^\sharp=\sum_{(r,s)\in O} h(r,s) \BK(r,s)$ where $h(r,s)\in \overline \Q$. There are  non-Galois cases such as $\{\left (\frac18,\frac 14\right), \left (\frac58,\frac 54\right),\left (\frac38,\frac 34\right)\}$. These pairs are conjugates in ${\mathbb S}_2$. The first two are in the same Hecke orbit, for example,
\begin{eqnarray*}
     \BK_2\left (1/8, 1/4\right)(16\tau)-8 \BK_2\left (5/8, 5/4\right)(16\tau)&=&f_{256.3.c.a}, \\ 
     \BK_2\left (1/8,1/4\right)(16\tau)+8 \BK_2\left (5/8, 5/4\right)(16\tau)&=&f_{256.3.c.b},
\end{eqnarray*}
while $ \BK_2\left (\frac38,\frac 34\right)(16\tau)$ is in the orbit of $f_{256.3.c.d}$. See Corollary 3.1 of \cite{HMM1}  for  the determination of $h(r,s)$ and the work by Rosen \cite{ENRosen} for more detailed discussion.

Using the idea and argument in the proof of Proposition 6.1 of \cite{HMM1}, one has
\begin{theorem}
    If $(r,s)\in\mathbb S_2$ is in a Galois orbit $O$,  then assumption (2) of \Cref{thm:main} is also satisfied for $\HD=\{\{\frac12,\frac12,r\},\{1,1,s\}\}$. 
\end{theorem} 
We will focus on the following family in this paper:
\begin{eqnarray}
    G2: & &\left(\frac{j}{24},\frac{24-j}{24}\right)\text{ for } j\in\{1,\cdots,11\}. \label{eq:G2}
\end{eqnarray}
Let $D=24/\gcd(j,24)$. For each $D$, the corresponding orbit is 
\[
    O(D)=\left\{\left(\frac iD, \frac{D-i}D+\left\lfloor \frac{2i}D \right\rfloor\right) \mid  i \in (\Z/D\Z)^\times\right\}.
\]
\begin{cor}\label{cor:G2-char}
    Every pair $(r,s)\in \mathbb S_2$ is in one of the G2 orbits if and only if $r+s=1$ or 2.
\end{cor}
In \Cref{tab:G2-eigenforms} below, we list the corresponding $f_{\{\frac12,\frac12,\frac j{24}\},\{1,1,\frac{24-j}{24} \}}^\sharp$ for the G2 family. As it only depends on $D$ not $j$, we will denote it by $f_{3,D}^\sharp.$
\subsection{Special \texorpdfstring{$L$}{L}-values of the \texorpdfstring{$\BK_2(r,s)$}{K2(r,s)} functions} \label{sec:K2-lvalues}
Below, we will relate the previous discussion to special $L$-values.  For a given weight-$k$ modular form, recall that
\begin{eqnarray}\label{eq:L(k)k=1,2}    
  L(f,1):=-2\pi i \int_0^{i\infty} f(t)dt,  \quad    L(f,2):=-4\pi^2 \int_0^{i\infty} f(t)tdt.
\end{eqnarray}
Also recall when $q_3>r_3$, 
\begin{equation*}
  \pPq32{r_1&r_2&r_3}{&q_2&q_3}{z} =\int_{0}^{1}t^{r_3-1}(1-t)^{q_3-r_3-1}   \pPq21{r_1&r_2}{&q_2}{tz}dt. 
\end{equation*}
Consequently, one has 
\begin{lemma}\label{lem:K-Lvalue}  
    Assume the notation as in Lemma \ref{lem:E-eta}, and let $N=N(r)$. Then
    \begin{eqnarray*}
        \pPq32{\frac 12&\frac 12&r}{&1&s}{1}
        &=& 2^{4r-1}N\pi\cdot L \left(\BK_2(r,s)( {N\tau}),1\right)\\
        &\overset{\eqref{eq:E-E*}}=& 2^{4s-4r-2}N^2 \cdot L\left(\BK_2(s-r,s)(N\tau),2\right).   
    \end{eqnarray*}
\end{lemma}
\begin{proof}
    Recall that the modular lambda function $\lambda(\tau)$ satisfies that, $\l(0)=1$ and $\l(i\infty)=0$ in the provided fundamental domain of $\G(2)$. We see that
    \begin{eqnarray*}
        \pPq32{\frac 12&\frac 12&r}{&1&s}{1}&=&\pi \int_0^1 {\l}^{r}{(1-\l)}^{s-r-1}   \pFq21{\frac 12&\frac 12}{&1}{\l}\frac{d\l}{\l}\\
        &\overset{\eqref{eq:alt-2-eta}}=& -\pi^2 i \cdot   2^{4r} \cdot \int_0^{i \infty} \BK_2(r,s)(\tau) d \tau\\
        &=&\pi 2^{4r-1} \cdot L\left(\BK_2(r,s)(\tau),1\right)\\
        &\overset{\tau\mapsto-1/\tau}=&2^{4s-4r-2}\cdot L\left(\BK_2(s-r,s)(\tau),2\right),
    \end{eqnarray*} 
    in the last step we apply the identity $\eta(-1/\tau) =\sqrt{\tau/i} \, \eta(\tau)$.  Our claims follow from the substitution $\tau\mapsto N\tau$. 
\end{proof}
\subsection{The \texorpdfstring{$\BK_1$-functions}{K1-functions}} 
The idea behind Theorem \ref{thm:main} can be used more generally. For example, when $\ba=\{\frac12\},\bbeta=\{1\}$, for $(r,s)\in\Q^2$ and $t=\lambda(\tau)$, \Cref{eq:f} gives rise to a weight two modular form  
\begin{equation}\label{eq:K1}
     \mathbb{K}_1(r,s)(\tau)=2^{1-4r}\left(\frac{\lambda}{1-\l}\right)^{r}(1-\lambda)^{s-1/2}\theta_3^4(\tau),
\end{equation}
for some finite index subgroup of $\mathrm{SL}_2(\Z).$ These $\mathbb{K}_{1}(r,s)(\tau)$ functions have also been considered by Rosen \cite{Rosen-K1}. Using an argument similar to \Cref{lem:K-Lvalue}, one has 
\begin{equation}\label{eq:K1-L}     
    L(\mathbb{K}_1(r,s)(\tau),1)=2^{1-4r}B(r,s-r-1/2).
\end{equation} In terms of eta function,  
\[
    \mathbb{K}_1(r,s)(\tau)=\frac{\eta(\tau/2)^{16s-8r-16}\eta(2\tau)^{8r+8s-12}}{\eta(\tau)^{24s-32}}.
\]
From which we know $\mathbb{K}_1(r,s)(\tau)$ is a congruence cusp form if $(r,s)$ is in  
\[
  \mathbb S_1=\{(r,s)\in\Q^2\mid 0<r<s<3/2,\, s-r>1/2, 24s\in\Z, 8(r+s)\in\Z\}.
\]
\subsection{Proof of \texorpdfstring{\Cref{thm:L-Values}}{Theorem 1.2}}
\begin{table}[ht]
    \centering
    \begin{tblr}{
        colspec={|Q[c,m]|Q[c,m]|Q[l,m]|},
        cell{3,6,9}{1-2}={r=2}{c},
        rowsep=3pt,
        hlines,
        hspan=minimal,
        hline{1,Z}={1pt},
        vline{1,Z}={1pt},
    }
        $D$ &$N$& \SetVline[1]{1-10}{black,1pt} \SetCell{c} Hecke Eigenform $f_{3,D}^\sharp$ in terms of $\BK_2(r,s)(N\tau)$ \\
        \SetHline[1]{1-3}{black, 1pt}
        $24$ &48 & \SetCell{c} {$f_{1152.3.b.i} = \BK_{2}\left(\frac{1}{24}, \frac{23}{24}\right) - \beta_{1}\BK_{2}\left(\frac{5}{24},\frac{19}{24} \right) - \beta_{5} \BK_{2}\left(\frac{7}{24}, \frac{17}{24} \right) - \beta_{3} \BK_{2}\left(\frac{11}{24},\frac{13}{24}\right)$ \\[1mm] 
        $-\beta_{2}\BK_{2}\left(\frac{13}{24},\frac{35}{24} \right) - \beta_{4}\BK_{2}\left(\frac{17}{24}, \frac{31}{24} \right) - \beta_{7} \BK_{2}\left(\frac{19}{24}, \frac{29}{24} \right) - \beta_{6} \BK_{2}\left(\frac{23}{24}, \frac{25}{24} \right)$}\\
        $12$ &24 & $f_{288.3.g.a} = \BK_{2}\left(\frac{1}{12},\frac{11}{12}\right) - 4 \BK_{2}\left(\frac{5}{12},\frac{7}{12} \right) + 4i \BK_{2}\left(\frac{7}{12}, \frac{17}{12} \right) - 16i \BK_{2} \left(\frac{11}{12}, \frac{13}{12} \right)$ \\
        &  & $f_{288.3.g.c} = f_{288.3.g.a}\otimes \chi_3, \qquad \chi_3 = \left(\frac{3}{\cdot}\right)$  \\ 
        $8$&16 & $f_{128.3.d.c} = \BK_{2}\left(\frac{1}{8}, \frac{7}{8} \right) + 2\sqrt{2}\BK_{2}\left(\frac{3}{8}, \frac{5}{8}\right) + 4i \BK_{2}\left(\frac{5}{8}, \frac{11}{8} \right) + 8 \sqrt{2}i \BK_{2}\left(\frac{7}{8},\frac{9}{8}\right)$ \\
        $6$& 12 & $f_{36.3.d.a}(\tau)+2f_{36.3.d.a}(2\tau)=\BK_2\left(\frac16, \frac56 \right)(12 \tau) + 8 \BK_2 \left(\frac56, \frac76 \right)(12 \tau) $  \\
        & & $f_{36.3.d.b}(\tau)-2f_{36.3.d.b}(2\tau)=\BK_2\left(\frac16, \frac56 \right)(12 \tau) - 8 \BK_2 \left(\frac56, \frac76 \right)(12 \tau) $ \\
        $4$&8 & $f_{32.3.c.a} = \BK_{2}\left(\frac{1}{4},\frac{3}{4}\right) + 4i \BK_{2}\left(\frac{3}{4},\frac{5}{4} \right)$ \\
        $3$&6  & $f_{36.3.d.a} = \BK_{2} \left(\frac{1}{3},\frac{2}{3} \right) - 2 \BK_{2} \left(\frac{2}{3},\, \frac{4}{3} \right)$ \\
        & &  $f_{36.3.d.b} = 
        f_{36.3.d.a}\otimes \chi_{-3}= f_{36.3.d.a}\otimes \chi_{3},\qquad \chi_{-3} = \left(\frac{-3}{\cdot}\right), \, \chi_{3} = \left(\frac{3}{\cdot}\right)$ \\
    \end{tblr}
    \caption{The $f^\sharp_{3,D}$ functions from G2 class.  In the above {\footnotesize $
    \beta_{1}=2i\sqrt{7},\, \beta_2=4i,\, \beta_3=4\sqrt2,\, \beta_4=8\sqrt 7, \, \beta_{5}=2i\sqrt{14},\, \beta_6=-16i\sqrt{2},\, \beta_7=8\sqrt{14},
    $} \\[2mm]
    and 
    $
        \BK_2(r,s)(\tau)= \frac{\eta \left(\tau/2 \right)^{16s-8r-12}\eta(2 \tau)^{8s+8r-12}}{\eta(\tau)^{24s-30}}
    $.}  
    \label{tab:G2-eigenforms}
\end{table}
\begin{table}[ht]
    \centering
    \begin{tblr}{
        colspec={|Q[c,m]|Q[c,m]|Q[l,m]|Q[c,m]|},
        rowsep=2pt,
        hlines,
        hspan=minimal,
        hline{1,Z}={1pt},
        vline{1,Z}={1pt},
    }
    $D$&$N$ & \SetCell{c} Hecke Eigenform $f_{2,D}^\sharp$   in terms of $\BK_1(r,s)(N\tau)$& CM \\ 
    \SetHline[1]{1-3}{black, 1pt}  
    24 &48& {$f_{1152.2.d.g}={\mathbb{K}_1(\frac1{24}, \frac{35}{24}) + 2\sqrt{3}i\, \mathbb{K}_1(\frac5{24}, \frac{31}{24})}$ \\ $+2 \sqrt{6} \,\mathbb{K}_1(\frac7{24}, \frac{29}{24}) - 4 \sqrt{2}i \,\mathbb{K}_1(\frac{11}{24}, \frac{25}{24})$} & $-24$ \\
     12 &24& $f_{288.2.a.a}=\mathbb{K}_1(\frac1{12}, \frac{17}{12}) - 4 \,\mathbb{K}_1(\frac{5}{12}, \frac{13}{12})$  & $-4$ \\
      8 &16& $f_{128.2.b.a}=\mathbb{K}_1(\frac18, \frac{11}8) + 2 \sqrt{2}i \,\mathbb{K}_1(\frac 38, \frac98)$ & $-8$ \\ 
    6 &6& $f_{36.2.a.a}=\mathbb{K}_1(\frac13,\frac76)$ & $-3$ \\
     4 & 8& $f_ {32.2.a.a}=\mathbb{K}_1(\frac14, \frac54) $ & $-4$ \\
     3 & 12&  $f_ {36.2.a.a}=\mathbb{K}_1(\frac16, \frac43) $ & $-3$ \\
    \end{tblr}
    \caption{ The $f_{2,D}^\sharp$ functions. $ \mathbb{K}_1(r,s)(\tau)=\frac{\eta(\tau/2)^{16s-8r-16}\eta(2\tau)^{8r+8s-12}}{\eta(\tau)^{24s-32}}$}
    \label{tab:f2D}
\end{table} 
The $f_{2,D}^\sharp$ functions have been computed by Esme Rosen in \cite{Rosen-K1} in terms of linear combinations of the $\BK_1(r,s)$ functions. They are listed in \Cref{tab:f2D}, in which we fix an embedding that $i^2=-1$.
\section{Four dimensional hypergeometric Galois representations}\label{sec:4dimGalois}
\subsection{Some well-poised formulae over \texorpdfstring{$\C$}{C} and \texorpdfstring{$\fq$}{Fq}}
Given a hypergeometric system in variable $\l$, the  variable change $\l\mapsto 1/\l$ will lead to another hypergeometric system playing a ``dual" role. See for example \cite{Beukers-Frederic} by Beukers--Jouhet. Note that well-poised hypergeometric local systems here are self-dual and hence the involution 
\begin{equation}\label{eq:iota}
    \mathfrak{i}: \l\mapsto 1/\l,
\end{equation} 
fixing $\l=\pm 1$, plays an important role  of the hypergeometric arithmetic  (cf. Proposition 2 of \cite{LLT2}).  For the cases in this paper, it can be realized by the map \eqref{eq:iota-map} acting on \eqref{eq:hgv}.  In \cite{Whipple25}, Whipple investigated well-poised series and very well-posed series at $\pm 1$ and derived a few reduction formulae. For this paper, putting two of them together give a more complete picture of \Cref{thm:mainclassical}. To make the notation compact, we use $r_{i}^{*} = 1 + r_{1} - r_{i}$. 
\begin{proposition}[Whipple]\label{prop:Whipple}
    Let $r_i\in\C$ such that the convergent conditions are satisfied. Then
    \begin{equation}\label{eq:Whipple-4F3&companion}
        \begin{split}
            \pFq43{r_1&r_2&r_3&r_4}{&r_2^*&r_3^*&r_4^*}{-1}&=C(r_i) \cdot \pFq32{1+\frac{r_1}2-r_2&r_3&r_4}{&1+\frac{r_1}2&r_2^*}1
            \\
            \pFq{5}{4}{r_1 & 1+\frac{r_1}{2}  & r_2 & r_3 & r_4}{& \frac{r_1}{2}  & r_2^* & r_3^* & r_4^*}{-1} 
        &=C(r_i)\cdot \pFq{3}{2}{\frac{1+r_1}2-r_2 & r_3 & r_4}{& \frac{1+r_1}2 & r_2^*}{1},
        \end{split}
    \end{equation}
    where $$C(r_i)=\frac{\G(r_3^*)\G(r_4^*)}{\G(1+r_1)\G(r_4^*-r_3)}.$$
\end{proposition}
Note that the second datum $\left\{\substack{r_1,\,1+\frac{r_1}2,\,r_2,\,r_3,\,r_4\\ 1, \,\,\, \,\frac{r_1}2,\,\,\,r_2^*,\,r_3^*,\,r_4^* }\right\}$  on the left hand side is imprimitive.
\begin{proof}
    The first one is (3.4) of \cite{Whipple25}.  To see the second one, we use (6.3) of \cite{Whipple25}, which is restated as
    \begin{singnumalign}\label{eq:Whipple-6F5}
        &\pFq{6}{5}{r_1 & 1+\frac{r_1}{2} & r_5 & r_2 & r_3 & r_4}{& \frac{r_1}{2} & r_5^* & r_2^* & r_3^* & r_4^*}{-1} \\
        &\hspace{1mm}= \frac{\G(r_3^*)\G(r_4^*)}{\G(1+r_1)\G(r_4^*-r_3)}\pFq{3}{2}{{r_5^*-r_2} & r_3 & r_4}{& r_5^* & r_2^*}{1}.
    \end{singnumalign} 
    When $r_5=\frac{1+r_1}2$, $r_5^*=r_5$ and hence the $_6F_5(-1)$ becomes the desired $_5F_4(-1)$.
\end{proof}
In comparison, we recall a  well-poised formula of McCarthy over finite fields. The notation $\chi = \square$ is used to indicate the character $\chi$ is a square.
\begin{theorem}[McCarthy \cite{McCarthy}]\label{thm:wp}
    Let $r_{1},r_{2},r_{3},r_{4}\in \Q$ such that all of the following data are primitive and $M$ is the least positive common denominator of the $r_{i}$'s. Let $q \equiv 1 \pmod{M}$ be an odd prime power. Then the following $H_{q}$ values are 0 if $\omega^{(q-1)r_{1}} \neq \square$ and when $\omega^{(q-1)r_{1}} = \square$,
    \begin{equation}\label{eq:McCarthy-4F3}
        H_{q}\left(\begin{matrix} r_{1}&r_{2}&r_{3}&r_{4}  \smallskip \\  1&r_{2}^*&r_{3}^*&r_{4}^* \end{matrix} \; ; \; -1 \right)=\displaystyle{\g_{\omega}\bigg(\frac{-r_{1},r_{3}-r_{4}^*}{-r_{3}^*,-r_{4}^*} \bigg) \sum_{ s\in\left\{ \frac{r_{1}}2,\frac{r_{1}+1}2\right\}} H_{q}\left(\begin{matrix} s-r_{2}&r_{3}&r_{4}  \smallskip \\  1&s&r_{2}^* \end{matrix} \; ; \; {1} \right)}.
    \end{equation}
\end{theorem}
In this formula, the right hand side contains two summands.  Comparing \eqref{eq:Whipple-4F3&companion} and \eqref{eq:McCarthy-4F3} closely, we see that not only  $\frac{\G(r_3^*)\G(r_4^*)}{\G(1+r_1)\G(r_4^*-r_3)}$ and $\g_{\omega}\bigg(\frac{-r_{1},r_{3}-r_{4}^*}{-r_{3}^*,-r_{4}^*} \bigg)$ match well up to applying reflection formulae, but also the two choices of $s$ correspond to the two $_3F_2(1)$'s on the left-hand sides of \eqref{eq:Whipple-4F3&companion} respectively. 
\subsection{Length four data from \texorpdfstring{$\BK_2$}{K2}-functions}
Letting 
\begin{equation}\label{eq:ri-specialized}
    r_1=r_2=r,\quad r_3=r_4=\frac12,
\end{equation} 
the right hand sides of  \Cref{prop:Whipple} are related to the  $\BK_2$-functions. In this case
\begin{equation}\label{eq:C-specialized}
    C(r)= \frac 1r\G\left(\frac{r+\frac12,r+\frac12}{r,r}\right). 
\end{equation}
\begin{remark}
For the hypergeometric differential equation associated to $\HD_4(r)$ in the variable $z$, when $r\neq 1/2$, $F(\HD_4(r);z)$ and $z^{1/2-r}F(\HD_4(1-r);z)$ are two independent solutions near 0. 
  
Like the discussion in \cite{WIN3a}, the local system of $\HD_4(r)$ at $\l=-1$  consisting of differential 3-forms of first and second kind on the variety \eqref{eq:hgv}. When $\l=-1$, these differentials admit the induced action of $ {\mathfrak {i}}$ given by \eqref{eq:iota-map}. For example, for $\omega=\frac{\mathrm{d} x_1\mathrm{d}x_2\mathrm{d}x_3}{y}$ then $ {\mathfrak {i}}^*\omega=u\omega$, for some root of unity in $\Z[\zeta_{24}]$. However, for differentials such as $\tilde{\omega}=(x_1-1/x_1)(x_2-1/x_2)(x_3-1/x_3)\omega$, $ {\mathfrak {i}}^*\tilde\omega=-u\tilde\omega$. %
For the arithmetic of $\HD_4(r)$ at $-1$,  $\BK_2(1-\frac r2,1+\frac r2)$ will play a crucial  role. Under the specialization \Cref{eq:ri-specialized}, the extra factor of $C(r)$ as above can be absorbed if we write the left hand side in the $P$-function notation  in  which corresponds  the order of elements in $\ba,\bbeta$ matters. 
\end{remark}
\subsection{Proof of \texorpdfstring{\Cref{thm:mainclassical}}{Theorem 1.1}}We now prove the first claim of this paper.
\begin{lemma}\label{lem:BetaProduct}
    For any real $r$,
    \begin{equation}\label{eq:4.7}
        B\left(1-\frac r2,r\right)B\left(\frac{1-r}2,r\right)=\frac{2^{2r}\pi}{r \sin(\pi r)} 
    \end{equation}
    when both hand sides are convergent.
\end{lemma}
\begin{proof}
    The standard properties of gamma and beta functions from \cite[Section 2]{Win3X} give
    \begin{equation*}
        \begin{split}
            B\left(1-\frac{r}{2},r\right)B\left(\frac{1-r}{2},r\right) &= \frac{\Gamma(1-\frac{r}{2})\Gamma(r)\Gamma(\frac{1-r}{2})\Gamma(r)}{\Gamma\left(\frac{1+r}{2}\right)\Gamma\left(1+\frac{r}{2}\right)} \\&= 2^{2r} \frac{\Gamma(r)\Gamma(-r)\Gamma\left(\frac{r}{2}\right)\Gamma\left(1-\frac{r}{2}\right)}{\Gamma\left(-\frac{r}{2}\right)\Gamma\left(1+\frac{r}{2}\right)}= \frac{2^{2r}\pi}{r \sin(\pi r)}.
        \end{split}
    \end{equation*}
\end{proof}
\begin{proof}[Proof of \Cref{thm:mainclassical}] 
    Below we will use $\displaystyle r=\frac{j}{12}.$  Let
    \begin{equation}\label{eq:K&E}
        K(r)(\tau)=\BK_1\left(\frac{1-r}2,1+\frac r2\right)(\tau), \quad E(r)(\tau)=\BK_1\left(\frac{-r}2,\frac{3+r} 2\right)(\tau),
    \end{equation}
    where $\BK_1(r,q)(\tau)$ is given in \eqref{eq:K1}. As $0<r<1/2$, $K(r)(\tau)$ is a holomorphic weight two cuspform while $E(r)(\tau)$ is not holomorphic. By  \Cref{eq:FtoP} 
    \begin{multline*}
        P(\HD_4(r);-1)=B(r,1-r) B(1/2,r)^2 F(\HD_4(r);-1)\\
        \overset{\eqref{eq:Whipple-4F3&companion}}=\frac{B(r,1-r)B(1/2,r)^2C(r)}{B(1/2,1/2)}\frac1{B(1-r/2,r)}\pPq32{\frac12&\frac12&1-\frac{r}2}{&1+\frac{r}2&1}1.
    \end{multline*}
    Applying \eqref{eq:C-specialized}, \eqref{eq:4.7}, and properties of gamma and beta functions, we get
    \[
        P(\HD_4(r);-1)=\frac{1}{2^{2r}} B\left(\frac{1-r}2,r\right)\cdot \pPq32{\frac12&\frac12&1-\frac{r}2}{&1&1+\frac{r}2}1.
    \]
    By \Cref{eq:K1-L} and \Cref{lem:K-Lvalue}, it agrees with
    \[
        2^{6r-2}\pi \cdot L\left(\BK_1\left(\frac{1-r}2,1+\frac r2\right)(\tau),1\right)\cdot L\left(\BK_2\left(1-\frac{r}2,1+\frac r2\right)(\tau),1\right).
    \]
    Similarly,
    \begin{equation*}
        \begin{split}    
            &\left(1+ \frac 2r \cdot z\frac{d}{dz} \right) \pPq{4}{3}{r&r&\frac{1}{2}&\frac{1}{2}}{&1&r+\frac{1}{2}&r+\frac{1}{2}}{-1}\\ 
            &=B(r,1-r)B(r,1/2)^2\pFq{5}{4}{r&1+\frac r2 &r&\frac{1}{2}&\frac{1}{2}}{&\frac r2&1&r+\frac{1}{2}&r+\frac{1}{2}}{-1}.
        \end{split}
    \end{equation*}
    Thus by \eqref{eq:Whipple-4F3&companion}, the left hand side equals
    \begin{equation*}
        \begin{split} 
            &-\frac{1}{2^{2r}} B\left(\frac{-r}2,r\right)\cdot \pPq32{\frac12&\frac12&\frac{1-r}2}{&1&\frac{1+r}2}1\\
            &= - {2^{6r-2}} \pi \cdot L\left(\BK_1\left(\frac{-r}2,\frac{3+r} 2\right)(\tau),1\right)\cdot L\left(\BK_2\left(\frac{1-r}2,\frac {1+r}2\right)(\tau),1\right).
        \end{split}
    \end{equation*}
\end{proof}
\subsection{Proof of  \texorpdfstring{\Cref{thm:mainGalois}}{Theorem 1.4}}  
The finite field analogue of \Cref{lem:BetaProduct} is as follows. 
\begin{lemma}\label{lem:JacProduct}
    Let $r\in \Q$ and $D$ be the denominator of $r/2$ and $q \equiv 1 \pmod{D}$ be an odd prime power. Then
    \begin{equation}\label{eq:4.9}
        J_\omega \left(1-\frac r2,r\right)J_\omega\left(\frac{1-r}2,r\right)=q \, \omega^{(q-1)r}(4)
    \end{equation}
\end{lemma}
\begin{proof}
    The standard properties of Gauss and Jacobi sums from \cite[Section 2]{Win3X} give
    \begin{equation*}
        \begin{split}
            J_\omega\left(1-\frac{r}{2},r\right)J_\omega\left(\frac{1-r}{2},r\right)&=\omega^{(q-1)r}(2) \frac{\g_\omega(r)\g(-r)\g_\omega\left(\frac{r}{2}\right)\g_\omega\left(1-\frac{r}{2}\right)}{\g_\omega\left(-\frac{r}{2}\right)\g_\omega\left(1+\frac{r}{2}\right)}\\
            &= q\omega^{(q-1)r}(-1) \, \omega^{(q-1)r}(4)  = q \, \omega^{(q-1)r}(4).
        \end{split}
    \end{equation*}
\end{proof}
We now use the Jacobi sums in the previous Lemma to describe the $p$th coefficients of the $f_{2,D}^\sharp$ functions, as listed in \Cref{tab:f2D}.
\begin{lemma}\label{lem:f2D}
    Let $j\in \{1,2,3,4,6\}$. For any prime $p\equiv 1\pmod {M:=\text{lcm}(4,D)}$ where $D=24/j$,     
    \begin{multline}\label{eq:4.10}
        a_p\left(f_{2,D}^\sharp \right) =-\P\left(\left\{\frac12,\frac j{24}\right\},\left\{1,\frac 32-\frac j{24}\right\};1;\wp\right)\iota_\wp(j/24)(1/16)\\
        -\P\left(\left\{\frac12,\frac {j+12}{24}\right\},\left\{1, 1-\frac j{24}\right\};1;\wp\right)\iota_\wp(j/24)(1/16)\\
       =-\left[J_\omega(j/24,1-j/12) +J_\omega((j+12)/24,-j/12)\right]\omega^{-(p-1)j/6}(2).
    \end{multline} 
    This means $\rho_{f_{2,D}^\sharp}|_{G(M)}$ is isomorphic to the direct sum of two characters of $G(M)$ given by the Jacobi sums.
\end{lemma}  
Before proving it,  we first mention a corollary. As $a_p\left(f_{2,D}^\sharp \right)\in\Z$, it is independent of the choice of $\omega \in \widehat{\F_p^\times}$ and how we embed $\omega$ into $\Z_p$.  Letting it be ${\bar \omega}_p$, by \Cref{cor:Jac-Gp} we obtain a description of the $p$-adic unit root of $a_p\left(f_{2,D}^\sharp\right)$.
\begin{corollary}\label{cor:ap(f2)decomposition}
From \eqref{eq:ap(f2D)modp}, for $j\in\{1,\cdots,6\}$, $r=j/12$, $D=24/\gcd(j,24)$ as before, for $p\equiv 1\mod D$
    \[
        a_p\left(f_{2,D}^\sharp \right)=\left(a_p\left(f_{2,D}^\sharp \right)\right)_0+p/\left(a_p\left(f_{2,D}^\sharp \right)\right)_0,
    \]
    where
    \begin{equation}\label{eq:mu-defn}
        \left(a_p\left(f_{2,D}^\sharp \right)\right)_0=\G_p\left(\frac {\frac r2,1-r}{1-\frac r2} \right) \omega_p^{(p-1)r}(4)=-(-1)^{(p-1)r/2} \omega_p(1/4)^{(p-1)r} \cdot \frac{\G_p(r)}{\G_p\left(\frac{1+r}{2}\right)^2}. 
    \end{equation} 
\end{corollary} 
Here is an outline of the next proof.  Using \Cref{thm:Weil}, we regard the sum of the two Jacobi sums as a 2-dimensional representation of $G(M)$ (which is extendable to $G_\Q$). Then we compare it with that arises from $f_{2,D}^\sharp$.  
\begin{proof}[Proof of \Cref{lem:f2D}] 
    First we explain the relation between the two $\P$-functions. In this proof we fix  the generator of the character group by assuming $\omega:=\iota_\wp(\frac 1{p-1})$, see \eqref{eqn:residue symbol} for the notation.  By the finite field Gauss summation formula (\cite[(6.11)]{Win3X}), 
    $$
      -\P\left(\left\{\frac12,\frac j{24}\right\},\left\{1,\frac 32-\frac j{24}\right\};1;\wp\right)\iota_\wp(-j/24)(16)=-J_\omega(j/24,1-j/12)\omega^{-\frac{(p-1)j}6}(2).
    $$ 
    By \Cref{thm:Weil}, it can be viewed as the Frobenius trace of a character $\chi_{D,1}$ of  $G(M)$. We will denote the above value by $u_{D,p}$ and $|u_{D,p}|=\sqrt p$ under any archimedean embedding. Similar conclusion applies to the second summand which is denoted by $v_{D,p}$ and use $\chi_{D,2}$ to denote the corresponding character  of $G(M)$. It is not isomorphic to $\chi_{D,1}$. Let  $\displaystyle \tau:\zeta_{M}\mapsto \zeta_{M}^{M/2+1}\in \Gal(\Q(\zeta_{M})/\Q)$.     
    It is straightforward to check that $\tau$ swaps $\chi_{D,1}$ and $\chi_{D,2}$. Use $\rho_D$ to denote $\chi_{D,1}\oplus \chi_{D,2}$ as a representation of $G(M)$. We will treat the right hand side of the claim as the Frobenius trace of $\rho_D$.
    
    Next we show that the claim holds when modulo $p$ under one embedding to $\Q_p$.  From \Cref{tab:f2D},  
    $$a_p\left(f_{2,D}^\sharp \right)=a_p\left(\BK_1 \left(\frac j{24},\frac 32-\frac j{24}\right)(2D\tau)\right)=a_p\left(\BK_1 \left(\frac 1{D},\frac 32-\frac 1{D}\right)(2D\tau)\right) $$ is an integer  bounded by $2\sqrt p$. We expand $(\l/16)^{1/D}(1-\l)^{-2/D}$ related to $\BK_1 (\frac 1{D},\frac 32-\frac 1{D})(\tau)$ as above  as a power series in terms of uniformizer $\l^{1/D}$. Using properties of  commutative formal group laws as in \cite[\S3.3]{HMM1}, the  $p^{th}$ coefficient of this power series, $4^{-2/D}\frac{\G(2/D+(p-1)/D)}{\G(2/D)\G(1+(p-1)/D)}$, is congruent to $a_p\left(f_{2,D}^\sharp \right)$ modulo $p$. Thus for $p\equiv 1\pmod M$
    5
    \begin{multline}\label{eq:ap(f2D)modp}
        a_p\left(f_{2,D}^\sharp \right)\equiv 4^{-2/D}\frac{\G(2/D+(p-1)/D)}{\G(2/D)\G(1+(p-1)/D)}\overset{\eqref{eq:gamma-to-p-gamma}}\equiv  \G_p\left(\frac{\frac1{D},1-\frac 2D}{1-\frac 1{D}}\right)4^{2(p-1)/D} \\ \overset{\Cref{cor:Jac-Gp}}\equiv -J_{\bar \omega_p}\left(\frac1{D},1-\frac 2D\right)4^{2(p-1)/D}\pmod p.
    \end{multline}
    The last claim follows from \Cref{cor:Jac-Gp}, by choosing $\omega$ as $\bar \omega_p$, the inverse of the Teichm\"uller character.  Thus for $p\equiv 1\pmod M$ under this embedding to $\Q_p$ 
    \[
        a_p\left(f_{2,D}^\sharp \right)\equiv \text{Tr}\rho_D(\Frob_p)=u_{D,p}+v_{D,p}\pmod p,
    \]
    as by the Gross--Koblitz formula, $J_{\bar \omega_p}\left(\frac1{D}+\frac12,1-\frac 2D\right)4^{2(p-1)/D}\in p\Z_p$. 
    
    Now we will apply the idea used in the proof of Proposition 6.1 of \cite{HMM1} to establish the equality. Pick an integral basis $b_1,\cdots,b_m$ for $\Z[\zeta_M]$ as a free-$\Z$-module. So $w_{D,p}:= u_{D,p}+v_{D,p}-a_p\left(f_{2,D}^\sharp \right)\in\Z[\zeta_M]$ can be written as $\sum_{i=1}^m a_ib_i$ where $a_i\in\Z$. For each $\sigma_k:\zeta_M\mapsto \zeta_M^k$ where $k\in (\Z/M\Z)^\times$, $\sigma_k(u_{D,p}+v_{D,p})$ consists of two other $\P$-functions with $j$ replaced by $kj$. Since for each $1\le i\le D/2$ coprime to $D$, $\BK_1(i/D,3/2-i/D)(2D\tau)$ is in the same Hecke orbit as $\BK_1(1/D,3/2-1/D)(2D\tau)$, so $\sigma_k(w_{D,p})\equiv 0 \pmod p$ for each $k$ as above. Thus $p\mid a_i$ for each $1\le i\le m.$ Meanwhile, under each archimedean embedding of $\Q(\zeta_M)$, the norm of $w_{D,p}$ is less than $4\sqrt p$. Thus for $p>17$, each $a_i=0$. Namely $w_{D,p}=0$. Smaller primes can be verified directly.  
    To verify that the two explicit $p$-adic expressions for $\left(a_p\left(f_{2,D}^\sharp\right)\right)_0$ agree, we use the reflection formula and a double angle formula for $\G_p$, see \eqref{eq:p-gamma-reflection} and \cite{CohenII}, respectively.  Using these gives
    \begin{multline*}
        \G_p\left(\frac {\frac r2,1-r}{1-\frac r2} \right) \omega_p^{(p-1)r}(4)=\G_p\left(\frac {(r+1)/2,\frac r2,1-r,r,1/2}{1-\frac r2,r,r/2,(r+1)/2,(r+1)/2} \right) \\=\G_p\left(\frac {r,1/2,1-r,r,1/2}{1-\frac r2,r,r/2,(r+1)/2,(r+1)/2} \right)\omega_p(1/4)^{(p-1)r}\\=- (-1)^{(p-1)r/2}
        \omega_p(1/4)^{(p-1)3r}\cdot \frac{\G_p(r)}{\G_p\left(\frac{1+r}{2}\right)^2}.
    \end{multline*}
\end{proof}
We are now ready to prove \Cref{thm:mainGalois} in a manner parallel to the proof of \Cref{thm:mainclassical}. 
\begin{proof}[Proof of \Cref{thm:mainGalois}]
    Let $\HD_3(r,s)=\{\{\frac12,\frac12, s-r\};\{1,1,s \}\}\}$. Note that by  \Cref{thm:main} and the discussion in \cite{HMM1}
    \begin{equation}\label{eq:4.11}
        \P(\HD_3(r,s);1,\wp) \iota_\wp(s-r)(1/16)=a_p(f_{3,D}^\sharp).
    \end{equation}
    Recall that $r=j/12$, $D=24/\gcd(j,24)$ is double the denominator of $r$ in the lowest case. We let $M=\text{lcm}(4,D)$ such that for any $\wp$ above  $p \equiv 1 \pmod{M}$, $-J(\iota_\wp(r),\iota_\wp(1-r))=1$ and $-J(\iota_\wp(\frac12),\iota_\wp(\frac12))=1$. Thus, by  \Cref{thm:Katz}, for any prime ideal $\wp$ above $p\equiv 1\pmod M$ $$\Tr \rho_{\{\HD_4(r);-1\}}(\text{Frob}_\wp)=-\mathbb{P}(\HD_{4}(r);-1; \wp).$$

    Using the conversion between the $\mathbb{P}$ and $H_{p}$ functions in \eqref{eq:P-H}, we get
    \begin{equation}\label{eq:PtoH}
        -\mathbb{P}(\HD_{4}(r);-1; \wp)
        \overset{\eqref{eq:P-H}}= J\left(\iota_\wp(r),\iota_\wp\left(\frac{1}{2}\right)\right)^{2} H_{p}(\HD_{4}(r);-1;\wp)
    \end{equation}
    Next we apply McCarthy's well-poised theorem \Cref{thm:wp}  to get
    \[
        -\mathbb{P}(\HD_{4}(r);-1; \wp)= p \left[H_{p}(\HD_{3}(r,r/2);1;\wp) + H_{p}(\HD_{3}((r,r+1)/2);1;\wp) \right].
    \]
    We continue by converting the $H_p$-functions back to $\P$-functions
    \begin{multline*}
        -\mathbb{P}(\HD_{4}(r);-1; \wp)=\\ -\frac{p}{J\left(\iota_{\wp}(r),\iota_{\wp}(- \frac{r}{2}\right))} \mathbb{P}\left(\HD_{3}\left(r,\frac r2\right);1;\wp\right) - \frac{p}{J\left(\iota_{\wp}(r),\iota_{\wp}(\frac{1-r}{2})\right)}\mathbb{P}\left(\HD_{3}\left(r,\frac{r+1}2\right);1;\wp\right).   
    \end{multline*}
    Finally, we finish the conclusion using the expressions for $a_{p}(f_{i,D}^{\#})$ for $i=2,3$
    \begin{equation*}
        \begin{split}
            &-\mathbb{P}(\HD_{4}(r);-1; \wp)\\
            &\overset{\eqref{eq:4.11}}= a_{p}(f_{3,D}^{\#})\iota_{\wp}(r)(1/4)\left[-\frac{p}{J\left(\iota_{\wp}(r),\iota_{\wp}(- \frac{r}{2}\right))} -\frac{p}{J\left(\iota_{\wp}(r),\iota_{\wp}(\frac{1-r}{2})\right)}\right] \\
            &\overset{\eqref{eq:4.9}}= a_{p}(f_{3,D}^{\#})\cdot \iota_{\wp}(r)(1/4)^2 \left[- J\left(\iota_{\wp}(r),\iota_{\wp}\left(\frac{1-r}{2}\right)\right) -J\left(\iota_{\wp}(r), \iota_{\wp}\left(\frac{-r}{2}\right)\right) \right]\\
            &\overset{\eqref{eq:4.10}}= a_{p}(f_{3,D}^{\#})a_{p}(f_{2,D}^{\#}).
        \end{split}
    \end{equation*}
The result now follows from \eqref{eq:PtoH} and equation (2.13) of \cite{Win3X}, which implies $$\Omega_{j,\F_q}=J\left(\iota_\wp(r),\iota_\wp\left(\frac{1}{2}\right)\right)^{2} =p\frac{\g(\iota_\wp(r))^2}{\g(\iota_\wp(r+1/2))^2}.$$  
In terms of representations, this means $\rho_{\{\HD(r);-1\}}$ and $\rho_{f_{2,D}^\sharp}\otimes \rho_{f_{3,D}^\sharp}$'s restrictions to $G(M)$ are isomorphic.
\end{proof}
We now consider two cases whose data are defined over $\Q$. By the work of Beukers, Cohen, and Mellit \cite{BCM}; the $H_p$ functions can be extended to almost all primes of $p$.
\begin{corollary}
    Consider the hypergeometric data $\HD_4(1/2)$, defined in \Cref{sec:intro}, and $\HD_{1} = \left\{\left\{\frac{1}{6},\frac{1}{6},\frac{5}{6},\frac{5}{6}\right\},\left\{\frac{1}{3},\frac{1}{3},\frac{2}{3},\frac{2}{3}\right\} \right\}$. Then for all odd primes we have
    \[
        H_p\left(\HD_{4}(1/2);-1\right)=a_p(f_{32.3.c.a})a_p(f_{32.2.a.a})
    \]
    and for all primes $p > 5$
    \[
        p^2H_p\left(\HD_{1};-1\right) =a_p(f_{288.3.g.c})a_p(f_{288.2.a.a}).
    \]
    Equivalently, if we let $\tilde \rho_{\left\{\HD_{4}(1/2);-1\right\}}$ and $\tilde \rho_{\left\{\HD_{1};-1\right\}}$ be  representations of $G_\Q$ such that at unramified primes $p$, 
    \[
        \tilde \rho_{\left\{\HD_{4} (1/2);-1\right\}}(\Frob_p)=H_p(\HD_{4}(1/2);-1) \text{ and } \tilde \rho_{\left\{\HD_1;-1\right\}}(\Frob_p)=p^2 H_p(\HD_1;-1).
    \]
    Then
    \[
        \tilde \rho_{\left\{\HD_{4}(1/2);-1\right\}} \cong \rho_{f_{32.3.c.a}} \otimes \rho_{f_{32.2.a.a}}
    \]
    and
    \[
        \tilde \rho_{\left\{\HD_{1};-1\right\}} \cong \rho_{f_{288.3.g.c}} \otimes \rho_{f_{288.2.a.a}} \cong \rho_{f_{288.3.g.a}} \otimes \rho_{f_{288.2.a.e}}.
    \]
\end{corollary} 
\begin{proof}
    For $\HD_{4}(1/2)$, by \Cref{thm:mainGalois}, we only need to consider primes $p\equiv 3 \pmod 4$. Then $H_p(\HD_1;-1)=0=a_p(f_{288.2.a.a})$, at primes $p \equiv 3 \pmod{4}$. The vanishing of $a_{p}(f_{288.2.a.a})$ follows since $f_{288.2.a.a}$ has CM by $\Q(\sqrt{-1})$. 
    
    Now $H_p(\HD_1;-1)=0$ at primes $p \equiv 3 \pmod{4}$ since
    \begin{align*}
        H_p\left(\HD_1;-1\right)
        =&\frac{1}{1-p}\sum_{\chi \in \widehat {\F_p^\times}}G(\chi)^2\chi(-1)
        \overset{\chi\mapsto \phi\ol\chi}= \frac{1}{1-p}\sum_{\chi \in \widehat {\F_p^\times}}G(\phi\ol \chi)^2\phi\ol \chi(-1) \\
        =& \phi(-1)  H_p\left(\HD_1;-1\right),
    \end{align*}
    where $\displaystyle G(\chi)=\ol \chi(2^4)\frac{\g(\chi^6)\g(\chi)\g(\ol \chi^3)}{\g(\chi^2)\g(\chi^3)\g(\ol \chi)}=\frac{\g(\phi \chi^3)\g(\ol \chi^3)}{\g(\phi\chi)\g(\ol \chi)}$, by \cite[Theorem 2.7]{Win3X}. It suffices to check primes $p \equiv 1 \pmod{12}$ together with an additional prime which is $5$ modulo 12, by the Chebotarev density theorem. Below we use $p=5$. 
    
    Fix $\omega$ a generator of $\widehat {\F_p^\times}$. The substitution $\chi \mapsto \chi {\omega}^{\frac{p-1}{3}}$ and \Cref{thm:wp}  give 
    \begin{align*}
        p^2\cdot   &H_p\left(\HD_2;-1\right)=J_w\left( 1/3, 1/6\right)^2\cdot H_p\left(\HD_4(1/6);-1\right)\\
        =&-\frac{J_w\left(1/3,  1/6\right)^2}{J_w\left( 1/2, 1/6\right)^2}\mathbb P\left(\HD_4(1/6);-1;p\right)=-\P\left(\HD_4(1/6),-1;p\right)\\
        =& -\left(J_w(5/6, 1/12) + J_w(5/6, 7/12)\right) \ol\omega^{\frac{p-1}{3}}(2)  a_{p}(f^\sharp_{3,12}).
    \end{align*}
    The multiplication formula for Gauss sums \cite[Theorem 2.7]{Win3X} and \Cref{lem:f2D} give
    \begin{equation*}
        \begin{split}
            & -\left(J_w(5/6, 1/12) + J_w(5/6, 7/12)\right) \ol\omega^{\frac{p-1}{3}}(2) \\&=-\ol\omega^{\frac{p-1}{4}}(-3)J_w(1/2,1/4)-\omega^{\frac{p-1}{4}}(-3)J_w(1/2,3/4)= a_{p}(f_{288.2.a.a}).
        \end{split}
    \end{equation*}
    Now \cite[Theorem 1.1]{HMM1} gives  $ a_{p}(f^\sharp_{3,12}) = a_{p}(f_{288.3.g.c}) = a_{p}(f_{288.3.g.a})$ for $p\equiv 1 \pmod{12}$. To decide the exact combination, we note that $f_{288.3.g.c}=f_{288.3.g.a}\otimes \chi_3$ and use $p=5$ to determine the forms by comparing 
    \begin{equation*}
        \begin{split}
            p^2\cdot H_p \left(\HD;-1\right)  =-16 &=4 a_{p}(f_{288.2.a.a})= a_{p}(f_{288.3.g.c})a_{p}(f_{288.2.a.a})\\
            &=-4 a_{p}(f_{288.2.a.e}) =a_{p}(f_{288.3.g.a})a_{p}(f_{288.2.a.e}).
        \end{split}
    \end{equation*}
    This gives the results. 
\end{proof}
\section{\texorpdfstring{$p$}{p}-adic aspects and the Proof of \texorpdfstring{\Cref{thm:mainpadic}}{cite}}\label{sec:padic}
\subsection{Some \texorpdfstring{$p$}{p}-adic background and the \texorpdfstring{$p$}{p}-adic perturbation method}\label{subsec:p-adic-background}
For the $p$-adic perspective, we introduce Morita's \cite{Mor-gamma} $p$-adic $\G_p$ function.
\begin{defn}\label{defn:p-gamma}
    The $p$-adic Gamma function $\G_p: \Z_p \to \Z_p^\times$ is defined on positive integers $n$ by
    \[
        \G_p(n) \colonequals (-1)^n \prod_{\substack{1 \leq i < n \\ p \nmid i}} i,
    \]
    and then extended continuously to $\Z_p$.
\end{defn}
This function satisfies properties similar to the classical $\G$ function, which can be quickly checked for positive integers and then extended by continuity to all of $\Z_p$.  First, we have an analog to the functional equation of the $\G$ function
\begin{equation}\label{eq:p-gamma-functional}
    \frac{\G_p\left(x+1\right)}{\G_p(x)} = \begin{cases} -x & \text{if } x \in \Z_p^\times \\
    -1 & \text{if } x \in p\Z_p. \end{cases}
\end{equation}
As an analog to Euler's reflection formula,
\begin{equation}\label{eq:p-gamma-reflection}
    \G_p(x)\G_p(1-x) = (-1)^{x_0},
\end{equation}
where $x_0$ is the unique solution to $x \equiv x_0 \pmod{p}$ with $1 \leq x_0 \leq p$. We will want to translate between classical $\G$ functions, Pochhammer symbols, and $p$-adic $\G_p$ functions.  To do so, we observe that $\G_p(n)$ is nearly $(n-1)!$ for $n \in \Z^+$, differing only by a sign and the missing multiples of $p$.  In particular,
\[
    (n-1)! = (-1)^{n} \G_p(n) \prod_{j=1}^{\lfloor \frac{n-1}{p}\rfloor} jp.
\]
This can then be extended continuously to general Pochhammer symbols, which allows us to rewrite quotients of $\G$-functions as quotients of $\G_p$-functions when the inputs differ by an integer by first expressing the quotient as a Pochhammer symbol.  See Lemma 17 of \cite{LR} for more details.  Doing this gives us the following:
\begin{lemma}
    Let $r \in \Q \cap \Z_p$, and let $0 \leq k \leq p-1$.  We have
    \begin{equation}\label{eq:poch-to-p-gamma}
        (r)_k = (-1)^k \frac{\G_p(r+k)}{\G_p(r)} (r+[-r]_0)^{\nu(k, [-r]_0}
    \end{equation}
    where $[\cdot]_0$ denotes the first $p$-adic digit and $\nu(k, x)$ is defined to be $0$ if $k \leq x$ and 1 if $k > x$.  Further, if $a-b \in \Z$, then
    \begin{equation}\label{eq:gamma-to-p-gamma}
        \frac{\G(a)}{\G(b)} = (-1)^{a-b} \frac{\G_p(a)}{\G_p(b)} v_{a,b,p},
    \end{equation}
    where $v_{a,b,p}$ is the product of those numbers $x$ between $a$ and $b$ for which $x-a \in \Z$ and $x \in p\Z_p$, or the reciprocals thereof in the case $b > a$.
\end{lemma}
We now consider the analytic properties of $\G_p$.  Define functions
\begin{equation}\label{eq:Gi-defn}
    G_i(x) \colonequals \frac{\frac{d^i}{dx^i}\G_p(x)}{\G_p(x)}.
\end{equation}
The following result of the second author and Ramakrishna allows us to approximate quotients of $p$-adic $\G$-functions in terms of these $G_i$ functions.
\begin{theorem}[Theorem 14, \cite{LR}]\label{thm:LR}
    Suppose $p \geq 5, r \in \N, a \in \Z_{p}, m \in \C_{p}$ satisfies $\nu_{p}(m) \geq 0$ and $t \in \{0,1,2\}$. Then 
    \[
        \frac{\Gamma_{p}\left(a+mp^{r}\right)}{\Gamma_{p}(a)} \equiv \sum_{k=0}^{t} \frac{G_{k}(a)}{k!} (mp^{r})^{k} \pmod{p^{(t+1)r}}.
    \]
    The above result holds for $t = 4$ if $p \geq 11$.
\end{theorem}
This result is the foundation of the \emph{$p$-adic perturbation} method.  The key idea of this method is to perturb the parameters of a hypergeometric datum by adding terms which are divisible by $p$.  When these perturbations are chosen cleverly, they can allow for better manipulations of the hypergeometric functions when viewed over the complex numbers or finite fields, while only giving small variations $p$-adically by \Cref{thm:LR}.  In the previous part of this series, we used rational functions to generate hypergeometric identities suitable to our purposes.  In this paper, we are able to exploit pre-existing hypergeometric identities, significantly simplifying the calculations.

Before moving to the supercongruences, we begin with an application of \Cref{thm:LR} to remove perturbations from quotients of $\G_p$-functions modulo high powers of $p$ by averaging over certain sums.
\begin{proposition}\label{prop:general-perturbation}
    Let $\ell \leq m$ be integers, $p$ be a prime, and let $\boldsymbol{\xi} = (\xi_1, \hdots, \xi_\ell, \hdots \xi_m) \subset \Z_p^m$.  Define
    \[
        C_{\ell}(\boldsymbol{\xi}) \colonequals \frac{\prod_{i=1}^\ell \G_p(\xi_i)}{\prod_{i=\ell+1}^m \G_p(\xi_i)}.
    \]
    Fix vectors $\boldsymbol{v} = (v_1, \hdots, v_m) \in \Z_p^m$ and $\boldsymbol{w} = (w_1, \hdots, w_n) \in \Z_p^n$ with $\boldsymbol{w}$ satisfying $\sum_{i=1}^n w_i = 0$ and $p \nmid n$.  Then
    \begin{equation}\label{eq:zeta-averaging-supercongruence}
        \frac{1}{n}\sum_{j=1}^n C_{\ell}(\boldsymbol{\xi}+w_j \boldsymbol{v}p) \equiv C_{\ell}(\boldsymbol{\xi}) \pmod{p^2}.
    \end{equation}
    Here we use $\boldsymbol{\xi}+w_j \boldsymbol{v}p$ to denote the vector $(\xi_1+ w_j v_1 p, \hdots, \xi_m + w_j v_m p)$.
\end{proposition}
\begin{remark}
    We only require the mod $p^2$ supercongruence in this paper, but we note that the below proof can be adapted to prove the supercongruence holds modulo $p^{\min\left\{n, 4\right\}}$ for $p \geq 11$.  Additionally, we have observed computationally that the supercongruence often holds modulo $p^n$.  Proving this using our approach would require a strengthening of \Cref{thm:LR}. 
\end{remark}
\begin{proof}
    First, we show that we can reduce to the case that $\ell=m$.  If $\ell < m$, then for each $\ell+1 \leq i \leq m$ the reflection formula \eqref{eq:p-gamma-reflection} for $\G_p$ gives us
    \[
        \frac{1}{\G_p(\xi_i)} = (-1)^{x_i}\G_p(1-\xi_i), \qquad \frac{1}{\G_p(\xi_i+w_j v_ip)} = (-1)^{x_i} \G_p(1-\xi_i-w_j v_ip),
    \]
    where $x_i$ is the unique integer in $\left\{1, \hdots, p\right\}$ such that $\xi_i \equiv x_i \pmod{p}$.  Define
    \begin{align*}
        &{\hat{\boldsymbol\xi}} = \left\{\xi_1, \xi_2, \hdots, \xi_\ell, 1-\xi_{\ell+1}, \hdots, 1-\xi_m\right\} \\
        &\widehat{\boldsymbol{\xi}+w_j\boldsymbol{v}p} = \left\{\xi_1+w_j v_ip, \hdots, \xi_\ell+w_j v_\ell p, 1-\xi_{\ell+1} - w_j v_{\ell+1}p, \hdots, 1-\xi_m - w_j v_m p\right\}.
    \end{align*}
    Assuming \eqref{eq:zeta-averaging-supercongruence} holds for $\ell=m$, then modulo $p^2$
    \[
        \frac{1}{n}\sum_{i=1}^n C_\ell(\boldsymbol{\xi}+w_j \boldsymbol{v}p) = \frac{(-1)^{\sum_{i=\ell+1}^m x_i}}{n} \sum_{i=1}^n C_m(\widehat{\boldsymbol{\xi}+w_j \boldsymbol{v}p}) \equiv (-1)^{\sum_{i=\ell+1}^m x_i} C_m(\hat{\boldsymbol{\xi}}) = C_\ell(\boldsymbol{\xi}).
    \]
    Thus, we need only prove the case $\ell=m$.  As such, we suppress the subscript and simply write $C$ to denote $C_m$ for the remainder of the proof.  For each $1 \leq j \leq n$ the definition of $C$ yields
    \[
        C(\boldsymbol{\xi}+w_j \boldsymbol{v}p) = \prod_{i=1}^m \G_p(\xi_i + w_j v_i p).
    \]
    Using \Cref{thm:LR}, each term of the product can be approximated $p$-adically by
    \[
        \G_p(\xi_i+w_j v_ip) \equiv \G_p(\xi_i)(1+G_1(\xi_i)w_jv_i p) \pmod{p^2}.
    \]
    Thus,
    \begin{align*}
        C(\boldsymbol{\xi}+w_j \boldsymbol{v}p) &\equiv \prod_{i=1}^n \G_p(\xi_i)(1+G_1(\xi_i)w_j v_ip) \pmod{p^2} \\
        &\equiv C(\boldsymbol{\xi})\left(1+ \sum_{i=1}^n G_1(\xi_i) w_j v_i p\right) \pmod{p^2}.
    \end{align*}
    Hence,
    \[
        \frac{1}{n} \sum_{j=1}^n C(\boldsymbol{\xi} +w_j \boldsymbol{v}p) \equiv C(\boldsymbol{\xi})\left(1 + \frac{1}{n} \sum_{1 \leq i, j \leq n} G_1(\xi_i)w_j v_i p\right) \pmod{p^2}.
    \]
    Thus, our desired congruence holds if
    \begin{equation}
        \frac{1}{n} \sum_{1 \leq i, j \leq n} G_1(\xi_i) w_j v_i \equiv 0 \pmod{p}.
    \end{equation}
    This follows from separating out the sum over the $w_j$'s, which we assumed to be zero.
\end{proof}
\begin{remark}
    As a particular example, suppose that $n \mid p-1$ so that there exists a primitive $n^{th}$ root of unity $\zeta_n \in \Z_p$.  Then we can take $w = (1, \zeta_n, \zeta_n^2, \hdots, \zeta_n^{n-1})$.
\end{remark}
\subsection{A Well-Poised Primitive \texorpdfstring{$_{4}F_{3}(-1)$}{4F3(-1)} Supercongruence}\label{ss:4F3sc}
We now turn our attention to the first supercongruence \eqref{eq:length-4-super} in Theorem \ref{thm:mainpadic}, which we restate below as \Cref{prop:A-supercongruence} for the reader's convenience.  Recall that, for $1 < j < 11$ we define $M = \mathrm{lcm}(4,D)$ for $D=24/\gcd(j,24)$, and
\[
    \HD_4\left(j/12\right) = \left\{\left\{\frac{j}{12}, \frac{j}{12}, \frac{1}{2}, \frac{1}{2}\right\}, \left\{1, 1, \frac{j}{12}+\frac{1}{2}, \frac{j}{12}+\frac{1}{2}\right\}\right\}.
\]
As $\HD_4(j/12)$ is not defined over $\Q$, the $p$-adic behavior is different for $1 \leq j \leq 6$ then it is for $6 < j < 12$.  In this section, we only consider the case $j \leq 6$.
\begin{proposition}\label{prop:A-supercongruence}
    Let $1 \leq j \leq 6$ and set $r = j/12$.  For $p\equiv 1\pmod M$, we have
   
    \[
       \Omega_{j,\Q_p} F\left(\HD_4\left(j/12\right); -1\right)_{p-1} \equiv \left(a_p(f_{2,D}^\sharp)\right)_0 a_p(f^\sharp_{3, D}) \pmod{p^2},
    \]
    where $\Omega_{j,\Q_p}=\frac{\G_p(j/12)^2}{\G_p(j/12+1/2)^2,}$ $f_{3, D}^\sharp$ is given in \Cref{tab:main} and $\left(a_p(f_{2,D}^\sharp)\right)_0$ is given in \Cref{eq:mu-defn}. 
\end{proposition}
We will obtain this by applying \Cref{prop:general-perturbation} to classical hypergeometric identities.  In particular, we will use the ${}_4F_3$ case of Whipple's formula \eqref{eq:Whipple-4F3&companion} as well as the following classical Kummer relation which holds whenever both sides converge.
\begin{equation}\label{eq:Kummer}
    \pFq{3}{2}{a&b&c}{&d&e}{1} = \frac{\Gamma(e)\Gamma(d+e-a-b-c)}{\Gamma(e-a)\Gamma(d+e-b-c)} \pFq{3}{2}{a&d-b&d-c}{&d&d+e-b-c}{1}.
\end{equation}
We first prove the following lemma using \Cref{prop:general-perturbation} which will be a necessary step in reducing $F(\HD_4(j/12); -1)$ to the case where Theorem 2.3 of \cite{HMM1} can be used to obtain the Fourier coefficients on the right-hand side of \Cref{prop:A-supercongruence}.
\begin{lemma}\label{lem:key1}
    Let $1 \leq j \leq 6$ and $r=j/12$.  If $p \equiv 1 \pmod{M}$, then modulo $p^2$,
    \begin{align*}
        p \pFq{3}{2}{\frac{1}{2}&\frac{1}{2}&1-\frac{r}{2}}{&1&1+\frac{r}{2}}{1}_{p-1} \equiv \frac{-\G_p\left(\frac{r}{2}\right)\G_p(r)}{\G_p\left(\frac{1+r}{2}\right)\G_p\left(r+\frac{1}{2}\right)} \pFq{3}{2}{\frac{1}{2}&\frac{1}{2}&\frac{r}{2}}{&1&r+\frac{1}{2}}{1}_{p-1}.
    \end{align*}
\end{lemma}
\begin{proof}
    With $x=\pm 1$, we take $a = (1-xp)/2$, $b=\left(1-\frac{r}{2}\right)(1+p)$, $c=1/2-xp$, $d=1-xp/2$, and $e=1+r/2$ in \eqref{eq:Kummer}.  This gives
    \begin{singnumalign}\label{eq:super-perturbation}
        &\pFq{3}{2}{\frac{1-xp}{2}&\left(1-\frac{r}{2}\right)(1+xp)&\frac{1}{2}-xp}{&1-\frac{xp}{2}&1+\frac{r}{2}}{1} \\
        &\hspace{5mm}= \frac{\G\left(1+\frac{r}{2}\right)\G\left(r+\frac{xpr}{2}\right)}{\G\left(\frac{1+r+xp}{2}\right) \G\left(r+\frac{1+xp(r-1)}{2}\right)} \pFq{3}{2}{\frac{1-xp}{2}&\frac{1+xp}{2}&\frac{r+(r-3)xp}{2}}{&1-\frac{xp}{2}&r+\frac{1+xp(r-1)}{2}}{1}.
    \end{singnumalign}
    First, we observe that the ${}_3F_2$ on the right-hand side terminates at $(p-1)/2$ for both choices of $x$, as we will have the negative integer $(1-p)/2$ appearing in the upper parameters in each case.  On the left-hand side, our hypergeometric series similarly terminates at $(p-1)/2$ for $x=1$, but does not terminate until $(p-1)(1-r/2)$ for $x=-1$.  Note that this index is an integer as the condition $p \equiv 1 \pmod{M}$ guarantees that the denominator of $r/2$ divides $p-1$.  One can quickly see that the coefficients of the series will be divisible by $p^2$ for $(p-1)/2 < k \leq (p-1)(1-r/2)$ and so the corresponding portion of the sum will vanish modulo $p^2$.  As such, we assume that our index $k$ lies in the range $[0..(p-1)/2]$ throughout. \par
    Beginning with the left-hand side of \eqref{eq:super-perturbation}, we use \eqref{eq:poch-to-p-gamma} to rewrite the coefficients of the series as
    \begin{singnumalign}\label{eq:LHS-3F2-p-adic}
        &\frac{\left(\frac{1-xp}{2}\right)_k \left(\left(1-\frac{r}{2}\right)\left(1+xp\right)\right)_k \left(\frac{1}{2}-xp\right)_k}{(1)_k \left(1-\frac{xp}{2}\right)_k \left(1+\frac{r}{2}\right)_k} \\
        &= -\G_p \left(\frac{\frac{1-xp}{2} + k, \left(1-\frac{r}{2}\right)\left(1+xp\right)+k, \frac{1}{2}-xp+k, 1-\frac{xp}{2}, 1+\frac{r}{2}}{\frac{1-xp}{2}, \left(1-\frac{r}{2}\right)\left(1+xp\right), \frac{1}{2}-xp, k+1, k+1-\frac{xp}{2}, k+1+\frac{r}{2}}\right)\\
        &\hspace{20mm} \times\left(\frac{rp}{2}\right)^{-\nu\left(k, \frac{r(p-1)}{2}-1\right)}.
    \end{singnumalign}
    In the notation of \Cref{prop:general-perturbation}, we have
    \begin{singnumalign}\label{eq:LHS-p-adic}
        &p \cdot \pFq{3}{2}{\frac{1-xp}{2}&\left(1-\frac{r}{2}\right)(1+xp)&\frac{1}{2}-xp}{&1-\frac{xp}{2}&1+\frac{r}{2}}{1} \\
        &\hspace{30mm} \equiv -p\sum_{k=0}^{\frac{p-1}{2}} C_5(\boldsymbol{\xi}^{(1)}(k)+\boldsymbol{v}^{(1)} xp)\left(\frac{rp}{2}\right)^{-\nu\left(k, \frac{r(p-1)}{2}-1\right)} \pmod{p^2},
    \end{singnumalign}
    where
    \begin{align*}
        \boldsymbol{\xi}^{(1)}(k) &\colonequals \left( \frac{1}{2}+k, 1-\frac{r}{2}+k, \frac{1}{2}+k, 1, 1+\frac{r}{2}, \frac{1}{2}, 1-\frac{r}{2}, \frac{1}{2}, k+1, k+1, k+1+\frac{r}{2} \right) \\
        \intertext{and}
        \boldsymbol{v}^{(1)} &\colonequals \left( \frac{-1}{2}, 1+\frac{r}{2}, -1, \frac{-1}{2}, 0, \frac{-1}{2}, 1-\frac{r}{2}, -1, 0, \frac{-1}{2}, 0\right).
    \end{align*}
    On the right-hand side, we begin with the quotient of $\G$-values.  By \eqref{eq:gamma-to-p-gamma}, we have
    \[
        \frac{\G\left(1+\frac{r}{2}\right)\G\left(r+\frac{xpr}{2}\right)}{\G\left(\frac{1+r+xp}{2}\right) \G\left(r+\frac{1+xp(r-1)}{2}\right)} = \frac{\G_p\left(1+\frac{r}{2}\right)\G_p\left(r+\frac{xpr}{2}\right)}{\G_p\left(\frac{1+r+xp}{2}\right)\G_p\left(r+\frac{1+xp(r-1)}{2}\right)} \cdot \frac{2}{pr}.
    \]
    The coefficients of the ${}_3F_2$ can be expressed in terms of $\G_p$ as before, which yields
    \begin{singnumalign}\label{eq:RHS-p-adic}
        p &\frac{\G\left(1+\frac{r}{2}\right)\G\left(r+\frac{xpr}{2}\right)}{\G\left(\frac{1+r+xp}{2}\right) \G\left(r+\frac{1+xp(r-1)}{2}\right)} \pFq{3}{2}{\frac{1-xp}{2}&\frac{1+xp}{2}&\frac{r+(r-3)xp}{2}}{&1-\frac{xp}{2}&r+\frac{1+xp(r-1)}{2}}{1} \\
        &\hspace{15mm}= -p \sum_{k=0}^{\frac{p-1}{2}} C_6\left(\boldsymbol{\xi}^{(2)}(k)+\boldsymbol{v}^{(2)} xp\right) \cdot \frac{2}{pr} \cdot \left(\frac{p[(r-3)x+r]}{2}\right)^{\nu\left(k, \frac{(p-1)r}{2}\right)}
    \end{singnumalign}
    where
    \begin{align*}
        \boldsymbol{\xi}^{(2)}(k) &= \left(1+\frac{r}{2}, r, \frac{1}{2}+k, \frac{1}{2}+k, \frac{r}{2}+k, 1, \frac{1+r}{2}, \frac{1}{2}, \frac{1}{2}, \frac{r}{2}, k+1, k+1, r+k+\frac{1}{2}\right) \\
        \intertext{and}
        \boldsymbol{v}^{(2)} &= \left(0, \frac{r}{2}, \frac{-1}{2}, \frac{r-3}{2}, \frac{1}{2}, \frac{-1}{2}, \frac{1}{2}, \frac{-1}{2}, \frac{1}{2}, \frac{r-3}{2}, 0, \frac{-1}{2}, \frac{r-1}{2} \right).
    \end{align*}
    In this case we have an equality rather than only a congruence modulo $p^2$ due to the fact that the ${}_3F_2$ on the right-hand side always terminates at $(p-1)/2$.  Combining \eqref{eq:LHS-p-adic} and \eqref{eq:RHS-p-adic}, we may reinterpret \eqref{eq:super-perturbation} $p$-adically as
    \begin{singnumalign}\label{eq:super-perturbation-p-adic}
        &-p\sum_{k=0}^{\frac{p-1}{2}} C_5(\boldsymbol{\xi}^{(1)}(k)+\boldsymbol{v}^{(1)}xp)\left(\frac{rp}{2}\right)^{-\nu\left(k, \frac{r(p-1)}{2}-1\right)} \\
        &\hspace{5mm} \equiv -p \sum_{k=0}^{\frac{p-1}{2}} C_6\left(\boldsymbol{\xi}^{(2)}(k)+ \boldsymbol{v}^{(2)} xp\right) \cdot \frac{2}{pr} \cdot \left(\frac{p[(r-3)x+r]}{2}\right)^{\nu\left(k, \frac{(p-1)r}{2}\right)} \pmod{p^2}.
    \end{singnumalign}
    For ease of notation, we let $\mathrm{LHS}(x)$ and $\mathrm{RHS}(x)$ denote the left and right-hand sides of \eqref{eq:super-perturbation-p-adic}, respectively, viewed as functions of $x$.  As the congruence holds for $x=\pm 1$, we can average both sides over the two values of $x$ as follows.  We have
    \begin{equation}\label{eq:averaged-supercongruence}
        \frac{\mathrm{LHS}(1)+\mathrm{LHS}(-1)}{2} \equiv \frac{\mathrm{RHS}(1) + \mathrm{RHS}(-1)}{2} \pmod{p^2}.
    \end{equation}
    On the left-hand side, we use \Cref{prop:general-perturbation} with $n=2$ to conclude
    \begin{align*}
        \frac{\mathrm{LHS}(1)+\mathrm{LHS}(-1)}{2} &= -p\sum_{k=0}^{\frac{p-1}{2}}\left(\frac{C_5(\boldsymbol{\xi}^{(1)}(k) + \boldsymbol{v}^{(1)} xp ) + C_5(\boldsymbol{\xi}^{(1) (k)} - \boldsymbol{v}^{(1)} xp )}{2}\right) \\& \hspace{10mm}   \times \left(\frac{rp}{2}\right)^{-\nu\left(k, \frac{r(p-1)}{2}-1\right)} \\
    \end{align*} 
    Thus, modulo $p^2$,
    \begin{singnumalign}\label{eq:LHS-simplification}
        \frac{\mathrm{LHS}(1)+\mathrm{LHS}(-1)}{2} &   \equiv -p \sum_{k=0}^{\frac{p-1}{2}} C_5(\boldsymbol{\xi}^{(1)}(k)) \left(\frac{rp}{2}\right)^{-\nu\left(k, \frac{r(p-1)}{2}-1\right)}\\
        &\equiv p \sum_{k=0}^{\frac{p-1}{2}} \G_p\left(\frac{\frac{1}{2}+k, 1-\frac{r}{2}+k, \frac{1}{2}+k, 1+\frac{r}{2}}{\frac{1}{2}, 1-\frac{r}{2}, \frac{1}{2}, k+1, k+1, k+1+\frac{r}{2}}\right)\left(\frac{2}{pr}\right)^{\nu\left(k, \frac{r(p-1)}{2}-1\right)} \\
        &\equiv p\sum_{k=0}^{\frac{p-1}{2}} \frac{\left(\frac{1}{2}\right)_k^2 \left(1-\frac{r}{2}\right)_k }{(1)_k^2\left(1+\frac{r}{2}\right)_k} \\
        &\equiv p \pFq{3}{2}{\frac{1}{2} & \frac{1}{2} & 1-\frac{r}{2}}{& 1 & 1+\frac{r}{2}}{1}_{p-1}.
    \end{singnumalign}
    The final congruence holds as the coefficients of this hypergeometric function are divisible by $p$ for $(p-1)/2 < k \leq p-1$. \par
    The right-hand side is more complicated, as we have
    \begin{align*}
        &\frac{\mathrm{RHS}(1) + \mathrm{RHS}(-1)}{2} \\
        &= \frac{-2}{r} \biggr( \sum_{k=0}^\frac{(p-1)r}{2} \frac{C_6(\boldsymbol{\xi}^{(2)}(k)+ \boldsymbol{v}^{(2)}p)+C_6(\boldsymbol{\xi}^{(2)}(k) -\boldsymbol{v}^{(2)}p)}{2} \\
        &\hspace{20mm}+\sum_{k=\frac{(p-1)r}{2}+1}^{\frac{p-1}{2}} \frac{\frac{p(2r-3)}{2} C_6(\boldsymbol{\xi}^{(2)(k)}+ \boldsymbol{v}^{(2)} p) +\frac{3p}{2}C_6(\boldsymbol{\xi}^{(2)}(k)- \boldsymbol{v}^{(2)}p)}{2} \biggr)
    \end{align*}
    For the first sum, we can simplify the inside to $C_6(\boldsymbol{\xi}^{(2)}(k))$ modulo $p^2$ as before.  For the second sum, we note that we only require a congruence modulo $p$ due to the additional multiple of $p$ appearing.  The congruences
    \[
        C_6(\boldsymbol{\xi}^{(2)}(k) \pm \boldsymbol{v}^{(2)}p) \equiv C_6(\boldsymbol{\xi}^{(2)}(k)) \pmod{p}
    \]
    are immediate consequences of \Cref{thm:LR}, and so 
    \begin{align*}
        &\sum_{k=\frac{(p-1)r}{2}+1}^{\frac{p-1}{2}} \frac{\frac{p(2r-3)}{2} C_6(\boldsymbol{\xi}^{(2)}(k)+\boldsymbol{v}^{(2)}p) + \frac{3p}{2}C_6(\boldsymbol{\xi}^{(2)}(k)-\boldsymbol{v}^{(2)} p)}{2} \\
        &\hspace{30mm} \equiv \sum_{k=\frac{(p-1)r}{2}+1}^{\frac{p-1}{2}} \frac{pr}{2}C_6(\boldsymbol{\xi}^{(2)}(k)) \pmod{p^2}. \\
    \end{align*}
    Therefore, modulo $p^2$ we have
    \begin{singnumalign}\label{eq:RHS-simplification}
        \frac{\mathrm{RHS}(1)+\mathrm{RHS}(-1)}{2} &\equiv \frac{-2}{r}\sum_{k=0}^{\frac{p-1}{2}} C_6(\boldsymbol{\xi}^{(2)}(k)) (pr/2)^{\nu\left(k, \frac{(p-1)r}{2}\right)} \\
        &\equiv \frac{2}{r} \sum_{k=0}^{\frac{p-1}{2}} \G_p\left(\frac{1+\frac{r}{2}, r, \frac{1}{2}+k, \frac{1}{2}+k, \frac{r}{2}+k}{\frac{1+r}{2}, \frac{1}{2}, \frac{1}{2}, \frac{r}{2}, k+1, k+1, r+k+\frac{1}{2}}\right)\left(\frac{pr}{2}\right)^{\nu\left(k, \frac{(p-1)r}{2}\right)}  \\
        &\equiv \frac{2}{r} \frac{\G_p\left(1+\frac{r}{2}\right)\G_p(r)}{\G_p\left(\frac{1+r}{2}\right)\G_p\left(r+\frac{1}{2}\right)} \sum_{k=0}^{\frac{(p-1)}{2}} \frac{\left(\frac{1}{2}\right)_k^2\left(\frac{r}{2}\right)_k}{k!^2 \left(r+\frac{1}{2}\right)_k} \\
        &\equiv \frac{-\G_p\left(\frac{r}{2}\right)\G_p(r)}{\G_p\left(\frac{1+r}{2}\right) \G_p\left(r+\frac{1}{2}\right)} \pFq{3}{2}{\frac{1}{2} & \frac{1}{2} & \frac{r}{2}}{&1 & r+\frac{1}{2}}{1}_{p-1}.
    \end{singnumalign}
    As with the left-hand side, we use the fact that the coefficients of this ${}_3F_2$ are divisible by $p^2$ for $(p-1)/2 < k \leq p-1$ to go from the truncation at $(p-1)/2$ to the truncation at $p-1$ in the congruence.  Combining \eqref{eq:averaged-supercongruence}, \eqref{eq:LHS-simplification}, and \eqref{eq:RHS-simplification} yields the claim.
\end{proof}
\begin{proof}[Proof of \Cref{prop:A-supercongruence}]
    We begin by taking $r_{1}=r_{2}=r,r_{3} = \frac{1+p}{2},$ and $r_{4} = \frac{1-p}{2}$ in \eqref{eq:Whipple-4F3&companion} which gives
    \begin{singnumalign}\label{eq:Whipple-4-specialized}
        &\pFq{4}{3}{r&r&\frac{1+p}{2}&\frac{1-p}{2}}{&1&r+\frac{1-p}{2}&r+\frac{1+p}{2}}{-1} \\
        &\hspace{30mm}= \frac{\Gamma\left(r+\frac{1-p}{2}\right)\Gamma\left(r+\frac{1+p}{2}\right)}{\Gamma\left(1+r\right)\Gamma\left(r\right)} \cdot \, \pFq{3}{2}{\frac{1+p}{2}&\frac{1-p}{2}&1-\frac{r}{2}}{&1&1+\frac{r}{2}}{1}.
    \end{singnumalign}
    Note that both sides terminate at $\frac{p-1}{2}$.  By \eqref{eq:gamma-to-p-gamma} and \eqref{eq:p-gamma-functional},
    \[
        {\frac{\G\left(r+\frac{1-p}{2}\right)\G\left(r+\frac{1+p}{2}\right)}{\G(1+r)\G(r)}} = \frac{-\G_p\left(r+\frac{1-p}{2}\right)\G_p\left(r+\frac{1+p}{2}\right)}{\G_p(r)^2} p.
    \]
    We observe that both the $\G_p$ quotient above as well as the coefficients of each hypergeometric function appearing in \eqref{eq:Whipple-4-specialized} are symmetric in $p$ and $-p$.  Thus, we may use \Cref{prop:general-perturbation} to remove the $p$-linear terms modulo $p^2$, giving
    \[
        F(\HD_4(r); -1)_{p-1} \equiv \frac{-\G_p\left(r+\frac{1}{2}\right)^2}{\G_p(r)^2} pr \cdot \pFq{3}{2}{\frac{1}{2}&\frac{1}{2}&1-\frac{r}{2}}{&1&1+\frac{r}{2}}{1}_{p-1} \pmod{p^2}.
    \]
    To apply the supercongruence Theorem 2.3 of \cite{HMM1}, we require the bottom right entry to lie between $1/2$ and $1$.  This is not the case for us, but we can use the Kummer relation \eqref{eq:Kummer} to transform to a ${}_3F_2$ for which the supercongruence result of \cite{HMM1} does apply.  Namely, by \Cref{lem:key1} we have
    \begin{align*}
        &F(\HD_4(r); -1)_{p-1} \equiv \frac{-\G_p\left(r+\frac{1}{2}\right)^2}{\G_p\left(r\right)^2 } \cdot p \, \pFq{3}{2}{\frac{1}{2}&\frac{1}{2}&1-\frac{r}{2}}{&1&1+\frac{r}{2}}{1}_{p-1} \\
        &\hspace{20mm}\equiv \frac{\G_p\left(r+\frac{1}{2}\right)\G_p\left(\frac{r}{2}\right)}{\G_p(r)\G_p\left(\frac{1+r}{2}\right)} \cdot \, \pFq{3}{2}{\frac{1}{2}&\frac{1}{2}&\frac{r}{2}}{&1&r+\frac{1}{2}}{1}_{p-1} \pmod{p^2}.
    \end{align*}
    Written this way, Theorem 2.3 of \cite{HMM1} then gives
    \begin{equation*}
        F(\HD_4(r); -1)_{p-1} \equiv \frac{\G_p\left(r+\frac{1}{2}\right)\G_p\left(\frac{r}{2}\right)}{\G_p(r)\G_p\left(\frac{1+r}{2}\right)} \cdot \, H_{p}\left[\begin{matrix} \frac{1}{2}&\frac{1}{2}&1-\frac{r}{2} \smallskip \\  &1&\frac{1}{2}-r \end{matrix} \; ; \; 1;\omega_p \right] \pmod{p^2}.
    \end{equation*} 
    To ease notation, we will drop $\omega_p$ from the character sum $H_p$ below. Next, we use the following special case of Greene's finite field Kummer transformation---see equation (4.25) of \cite{Greene}
    \[
        H_{p}\left[\begin{matrix} \frac{1}{2}&\frac{1}{2}&1-\frac{r}{2} \smallskip \\  &1&\frac{1}{2}-r \end{matrix} \; ; \; 1 \right] = \omega_p^{(p-1)(r/2)}(-1) \frac{J_{\omega_p}(-r/2,r)}{J_{\omega_p}(-r/2,1/2-r/2)}  H_{p}\left[\begin{matrix} \frac{1}{2}&\frac{1}{2}&-\frac{r}{2} \smallskip \\  &1&\frac{r}{2} \end{matrix} \; ; \; 1 \right]. 
    \]
    The hypergeometric datum on the right-hand side is $\HD_3(r/2, r)$, where $\HD_3$ is as in \eqref{eq:4.11}. Modulo $p^2$ we then deduce
    \begin{align*}
        F(\HD_4(r); -1)_{p-1} &\equiv (-1)^{\frac{(p-1)r}{2}} \frac{\G_p\left(r+\frac{1}{2}\right) \G_p\left(\frac{r}{2}\right)}{\G_p(r)\G_p\left( \frac{1+r}{2} \right)} \frac{J_{\overline{\omega}_p} \left(\frac{r}{2}, 1-r\right)}{J_{\overline{\omega}_p}\left(\frac{r}{2}, \frac{1+r}{2}\right)} H_p \left(\HD_3\left(\frac{r}{2}, r\right); 1\right) \\
        &\hspace{-1mm} \stackrel{\eqref{eq:P-H}}{\equiv} (-1)^{\frac{(p-1)r}{2}} \frac{\G_p\left(r+\frac{1}{2}\right)^2}{\G_p\left(\frac{1+r}{2}\right)^2\G_p(r)} \mathbb{P}\left(\HD_3\left(\frac{r}{2}, r\right), 1, p\right) \\
        &\hspace{-2mm}\stackrel{\eqref{eq:4.11}}{\equiv} (-1)^{\frac{(p-1)r}{2}} \omega_p(1/4)^{(p-1)r} \frac{\G_p\left(r+\frac{1}{2}\right)^2}{\G_p\left(\frac{1+r}{2}\right)^2\G_p(r)} a_p(f_{3,D}^{\sharp}).
    \end{align*}
    Recalling $\Omega_{j, \Q_p} = \G_p(r)^2/\G_p(r+1/2)^2$ where $r=j/12$ we conclude that
    \begin{align*}
        \Omega_{j,\Q_p} {}F(\HD_4(r); -1)_{p-1} &\equiv (-1)^{\frac{(p-1)r}{2}} \omega_p(4)^{(1-p)r} \frac{\G_p(r)^2}{\G_p\left(r+\frac{1}{2}\right)^2} \frac{\G_p\left(r+\frac{1}{2}\right)^2}{\G_p\left(\frac{1+r}{2}\right)^2\G_p(r)} a_p(f_{3,D}^{\sharp}) \\
        & \equiv (-1)^{\frac{(p-1)r}{2}} \omega_p(1/4)^{(p-1)r} \frac{\G_p\left(r\right)}{\G_p\left(\frac{1+r}{2}\right)^2} a_p(f_{3,D}^{\sharp})
        \\ 
        &\hspace{-2mm}\overset{\eqref{eq:mu-defn}} \equiv \left(a_p(f_{2,D}^\sharp\right)_0 a_p(f_{3,D}^{\sharp})\pmod{p^2}.
    \end{align*}
\end{proof} 
\subsection{A Very Well-Poised Imprimitive \texorpdfstring{$_{5}F_{4}(-1)$}{5F4(-1)} Supercongruence}\label{ss:5F4sc}
We now turn our attention to the second congruence \eqref{eq:length-5-super} of \Cref{thm:mainpadic} which we restate in the following proposition:
\begin{proposition}\label{prop:super2}
    For any prime $p \equiv 1 \pmod{M}$,
    \[
        \Omega_{j,\Q_p}F(\HD_5(j/12);-1)_{p-1} \equiv
        \left(a_p(f_{2,D}^\sharp)\right)_1a_p(f_{3,D}^\sharp)\pmod{p^2}
    \]
    where $\left(a_p\left(f_{2,D}^\sharp\right)\right)_1 = p/\left(a_p\left(f_{2,D}^\sharp\right)\right)_0$.
\end{proposition}    
\begin{proof}
    We begin by substituting
    \[ 
        r_1=r_2=r, \quad r_3=\frac{1+p}{2}, \quad r_4=\frac{1-p}{2}
    \]
    into \eqref{eq:Whipple-4F3&companion}.  This gives us
    \begin{equation*}
    \begin{split}
        &\pFq{5}{4}{r&1+\frac{r}{2}&r&\frac{1+p}{2}&\frac{1-p}{2}}{&\frac{r}{2}&1&r+\frac{1-p}{2}&r+\frac{1+p}{2}}{-1}\\
        &\hspace{50mm}= \frac{\G\left(r+\frac{1-p}{2}\right)\left(r+\frac{1+p}{2}\right)}{\G\left(1+r\right)\G\left(r\right)} \pFq{3}{2}{\frac{1-r}{2}&\frac{1+p}{2}&\frac{1-p}{2}}{&\frac{1+r}{2}&1}{1}.
    \end{split}
    \end{equation*}
    Under this perturbation, both hypergeometric series now terminate at $(p-1)/2$.  As in the previous proposition, the $\G$-quotient can be expressed $p$-adically as
    \begin{equation}\label{eq:Whipple-Gamma-simplified}
        \frac{\G\left(r+\frac{1-p}{2}\right)\G\left(r+\frac{1+p}{2}\right)}{\G(1+r)\G(r)} = \frac{-\G_p\left(r+\frac{1-p}{2}\right)\G_p\left(r+\frac{1+p}{2}\right)}{\G_p(r)^2} p.
    \end{equation}
    As in the previous proof, we may use \Cref{prop:general-perturbation} to remove the $p$-linear terms on both sides modulo $p^2$ before using Theorem 2.3 of \cite{HMM1} and \eqref{eq:2.9} to relate to the corresponding hypergeometric character sum.  Doing so gives us
    \begin{align*}
        \Omega_{j, \Q_p}F(\HD_5(r); -1) &= -p \, \pFq{3}{2}{\frac{1-r}{2}&\frac{1}{2}&\frac{1}{2}}{&\frac{1+r}{2}&1}{1} \pmod{p^2} \\
        &\equiv-p H_p \left[ \begin{matrix} \frac{1}{2} & \frac{1}{2} & \frac{1+r}{2} \\ & 1 & \frac{1-r}{2} \end{matrix} \; ; \; 1\right] \pmod{p^2} \\
        &\equiv \frac{-\G_p\left(\frac{1+r}{2}\right)}{\G_p(r)\G_p\left(\frac{1-r}{2}\right)}p \, \mathbb{P}\left(\HD_3\left(1-r, \frac{1-r}{2}\right), 1, p\right) \pmod{p^2} \\ 
        & \equiv -(-1)^{\frac{(p-1)r}{2}}\frac{\G_p\left(\frac{1+r}{2}\right)^2}{\G_p(r)}p \, \omega_p(4)^{(p-1)r} a_p\left(f_{3,D}^\sharp\right) \pmod{p^2} \\
        & \equiv \left(a_p(f_{2,D}^\sharp)\right)_1 a_p\left(f_{3,D}^\sharp\right) \pmod{p^2},
    \end{align*}
    as was to be shown.
\end{proof}
\begin{remark}
\begin{enumerate}
    \item In the $r=1/2$ case, the approach of \S 2.5 of \cite{Long} can be used to show that this congruence in fact holds modulo $p^3$.  The congruence does not hold beyond $p^2$ for the other cases.
    \item In Part I of  the series \cite[(5.29)]{HMM1}, we introduced the hypergeometric sum $E_{\mathrm{GK}}(\HD;\l)$ as the potential obstruction to a mod $p^2$ supercongruence.  With this, the supercongruence results in this paper can be rephrased as, for $p\equiv 1\pmod M$, 
    \begin{equation}
        E_{\mathrm{GK}}(\HD_4(j/12);-1)\equiv  \Omega_{j,\Q_p}^{-1} \cdot\left(a_p\left(f_{2,D}^\sharp\right)\right)_1 a_p\left(f_{3,D}^\sharp\right)/p \pmod p.
    \end{equation}
\end{enumerate}    
\end{remark}
\section{Appendix}
\subsection{Some \texorpdfstring{$L$}{L}-values of weight three CM modular forms in terms of \texorpdfstring{${}_3F_2(1)$}{3F2(1)}}
In the literature, hypergeometric values have been used to compute special $L$-values of modular forms and other arithmetic objects, such as Mahler measures of two-variable Laurent polynomials, as in \cite{Rogers-Zudilin} by Rogers--Zudilin. In this subsection, we record the special $L$-values of certain cuspforms related to the ${}_3F_2(\HD_d^{(3)};1)$ values, where $\HD_d^{(3)}:=\{\{\frac 12, \frac 1d,1-\frac1d\},\{1,1,1\}\}$ for $d=2$, $3$, $4$, and $6$.  The CM cases of the G2 family not in this form will be treated in the next subsection.
\begin{theorem}
    \[
        \pFq32{\frac 12&\frac12&\frac12}{&1&1}1 = \frac 2{\sqrt 3}\pFq32{\frac 12&\frac16&\frac56}{&1&1}1  =\frac 8{\pi} L(\eta(4\tau)^6,1)= \frac {16}{\pi^2} L(\eta(4\tau)^6,2);
    \]
    \[
    \pFq32{\frac 12&\frac13&\frac23}{&1&1}1 = \frac{6 \sqrt{3}}{\pi} L(\eta(2\tau)^3\eta(6\tau)^3,1)=\frac{18}{\pi^2} L(\eta(2\tau)^3\eta(6\tau)^3,2).
    \]
    Let $f_{8.3.d.a}(\tau)=\eta(\tau)^2\eta(2\tau)\eta(4\tau) \eta(8\tau)^2$. 
    \begin{align*}
        \pFq32{\frac 12&\frac14&\frac34}{&1&1}1 = &\sqrt 2\pFq32{\frac 12&\frac12&\frac12}{&1&1}{-1}\\
        =& \frac {12\sqrt 2}{\pi} L(f_{8.3.d.a}(\tau),1)= \frac {24}{\pi^2} L(f_{8.3.d.a}(\tau),2).
    \end{align*}
\end{theorem}
\begin{proof}  
    By a Bailey cubic hypergeometric transformation \cite[pp. 185]{AAR}, we have
    \[
        \pFq32{\frac 12&\frac16&\frac56}{&1&1}1 =\frac{\sqrt 3}2 \pFq32{\frac 12&\frac12&\frac12}{&1&1}1. 
    \]
    The result follows as stated in \Cref{sec:K2-lvalues}.  In the work of Osburn and Straub \cite{Osburn-Straub}, they obtain the special value
    \[
        \pFq32{\frac 12&\frac12&\frac12}{&1&1}{-1}= \frac {12}{\pi} L(f_{8.3.d.a}(\tau),1).
    \]
    By a quadratic transformation  (or identities between the corresponding modular forms, see \cite[Appendix]{HMM1} for example), one has
    \begin{multline*}
        \pFq32{\frac 12&\frac14&\frac34}{&1&1}1 =\sqrt 2\pFq32{\frac 12&\frac12&\frac12}{&1&1}{-1}\\=
        \frac {12\sqrt 2}{\pi} L(f_{8.3.d.a}(\tau),1)= \frac {24}{\pi^2} L(f_{8.3.d.a}(\tau),2).
    \end{multline*}
    For the  $d=3$ case, see \cite{DMgrove} for detail. 
\end{proof}
\subsection{From modular forms to new \texorpdfstring{$_3P_2(1)$}{3P2(1)} identities} \label{ss:P-identities}
Assume $(r,s)$ is in a Galois $\BK_2$ orbit. As a byproduct of Lemma \ref{lem:K-Lvalue} and the fact that $\{\BK_2(r_c,q_c)\}$ and  $\{\BK_2^\ast(r_c,q_c)\}$  are two bases for the corresponding invariant subspaces of the Hecke operators, one has identities among the ${}_3F_2(1)$-values, which may be induced from the classical hypergeometric identities. In the following, we provide an immediate corollary of \eqref{eq:Kummer} and less trivial types of identities. For a more detailed discussion and other types of relations, please see the forthcoming work by Rosen \cite{ENRosen}. 

For a shorthand notation, we let $P(r,s)$ denote the period value 
\begin{equation}\label{eq:P(r,s)}
     P(r,s):=\pi B(r,s-r)\pFq32{\frac 12&\frac12&r}{&1&s}1. 
\end{equation}
Using this notation, a specialized version of \Cref{eq:Kummer} is as follows.
\begin{equation}
    \label{cor:AAR3.3.5}
      P(r,s) =\sin(\pi r)\frac{B(s-r,r)^2}{B(s-1/2,1/2)^2} P(1-r,s-r+1/2).
\end{equation}
We first consider the cases when $m=4$ and $8$, which are CM cases.
\begin{proposition}\label{prop: ALids-m=8}  
    When $m=4$ and $8$, we have
    \[
        \BK_{2} \left(1/3,2/3 \right)= \BK_{2}^* \left(1/3,2/3  \right) \mbox{ and }  \BK_{2} \left(2/3,4/3  \right)= \BK_{2}^* \left(2/3,4/3 \right).
    \]
    \begin{align*}
        \BK_{2} \left(1/6,5/6 \right)(12\tau)=&  \BK_{2} \left(1/3,2/3 \right)(6\tau)-4 \BK_{2} \left(2/3,4/3 \right)(12\tau)\\
        \BK_{2}^\ast \left(1/6,5/6\right)(6\tau)=& \BK_{2} \left(2/3,5/6 \right)(6\tau) = 2\BK_{2} \left(1/3,2/3 \right)(12\tau)-\BK_{2} \left(2/3,4/3 \right)(6\tau)\\
        \BK_{2}^\ast \left(5/6,7/6\right)(6\tau)=& \BK_{2} \left(1/3,7/6 \right)(6\tau) = \BK_{2} \left(1/3,2/3 \right)(6\tau)-8\BK_{2} \left(2/3,4/3 \right)(12\tau).
    \end{align*}
    The corresponding $P$-values are algebraically dependent as follows. 
    \begin{align*}
        P\left(1/6,5/6\right)=&2^{-5/6}3^{-1/4}B(1/4,1/4)^2=2^{1/6}3^{-1/4}B(1/4,1/2)^2,\\
        2^{\frac 13}P\left(1/3,2/3\right) =&  (2+\sqrt 3)   P\left(2/3,4/3\right), \\
        2^{\frac 13}P\left(5/6,7/6\right) =&  P\left(1/6,5/6\right) =(1+\sqrt 3) P\left(2/3,4/3\right)   ,\\
        2^{\frac 13} P\left(1/3,7/6\right) =&   (2-\sqrt 3) P\left(2/3,5/6\right)= \sqrt 3 P\left(2/3,4/3\right). 
    \end{align*}
\end{proposition}
\begin{proof}
    Comparing the $q$-expansions, we have 
    \[
        \BK_{2} \left(1/6,5/6 \right)(12\tau)=  \BK_{2} \left(1/3,2/3 \right)(6\tau)-4 \BK_{2} \left(2/3,4/3 \right)(12\tau)
    \]
    and hence
    \begin{equation*}
        \begin{split}
            L\left(\BK_{2} \left(1/6,5/6 \right)(12\tau) , 1\right) =& L\left( \BK_{2} \left(1/3,2/3 \right)(6\tau), 1\right)-4L\left( \BK_{2} \left(2/3,4/3 \right)(12\tau), 1\right)\\
            =&L\left( \BK_{2} \left(1/3,2/3 \right)(6\tau), 1\right)-2L\left( \BK_{2} \left(2/3,4/3 \right)(6\tau), 1\right). 
        \end{split}
    \end{equation*}
    Other identities can be obtained in a similar way. 
    
    By Lemma \ref{lem:K-Lvalue}, \Cref{cor:AAR3.3.5}, $2^{\frac 16}\G(1/3)\G(5/6)=\sqrt{2\pi}\G(2/3)$ which follows from the duplication formula of $\G(\cdot)$, and $\pi=\G(1/3)\G(2/3)\sin(\pi/3)$ from the reflection formula of $\G(\cdot)$, we derive the algebraic relations 
    \[
        2^{\frac 13}P\left(1/3,2/3\right) =  (2+\sqrt 3)   P\left(2/3,4/3\right)
    \]
    and 
    \[
        P\left(1/6,5/6\right) = 2^{1/3}P(1/3,2/3)-P(2/3,4/3)= (1+\sqrt 3) P\left(2/3,4/3\right).
    \]
    To evaluate $P(1/6,5/6)$, we first use \eqref{eq:Kummer} and then apply Dixon's evaluation \cite[(2.2)]{Whipple25} and properties of gamma function.
    \begin{multline*} 
        \pFq{3}{2}{\frac12&\frac12&\frac16}{&\frac56&1}{1} = \frac{\Gamma(\frac23)}{\Gamma(\frac12)\Gamma(\frac76)} \pFq{3}{2}{\frac12&\frac13&\frac23}{&\frac76&\frac56}{1}
        =\frac{2^{3/2}}{3^{1/4}\pi^2}\G\left( \frac{\frac23,\frac14,\frac14,\frac14,\frac14}{\frac12, \frac16}\right),
    \end{multline*}
    which is equivalent to our claim. 
\end{proof}
The $P\left(1/3,2/3\right)$-value can also be obtained by \cite[Theorem 3.5.5(i)]{AAR} with $c=1/3$  followed by the reflection and multiplication formulae of $\G$-function. 

For other cases, we have the following relations between $\BK_2^\ast(m/24,1-m/24)=\BK_2(1-m/12,1-m/24)$ and $\BK_2(mj/24, 1-mj/24)$ in the same Hecke space, which lead to the $P$-identities. 
\begin{proposition}\label{prop: ALids} 
    For a fixed $m=1$, $2$, $3$, let $j \in \left(\Z/(12/m)\Z\right)^\times$.
    \begin{multline*}
        \BK_2\left(\frac{mj}{12},\frac 34 \pm \frac14+\frac{mj}{24}\right)\left(\frac{24}m\tau\right) = \\ \BK_2\left(\frac{mj}{24},1-\frac{mj}{24}\right)\left(\frac{48}m\tau\right)\mp 4  \BK_2\left(\frac 12+\frac{mj}{24},\frac32-\frac{mj}{24}\right)\left(\frac{48}m\tau\right),
    \end{multline*}
    here we use $\mp$ to denote the opposite sign choice from $\pm$.  In terms of $P$-values, 
    \[
        2^{1-mj/6}P\left(\frac{mj}{12},\frac34\pm\frac 14+\frac{mj}{24}\right) = P\left(\frac{mj}{24},1-\frac{mj}{24}\right)\mp P\left(\frac 12+\frac{mj}{24},\frac32-\frac{mj}{24}\right). 
    \]
\end{proposition}
\begin{proof} Below we let $t(\tau)$ be the Hauptmodul on $\G_0(8)$ such that $$\l(2\tau)=16t(\tau)(1-4t(\tau)).$$
    The identities between the modular forms follow from our construction and the classical results
    \[
        2^{4r-1}\BK_2(r,s)(\tau)=\left(\frac\l{1-\l}\right)^r(1-\l)^s \theta_3^6(\tau),
    \]
    \[
        \frac{\l(\tau)^2}{1-\l(\tau)}= \frac{16\l(2\tau)}{(1-\l(2\tau))^2}  \quad \mbox{ and }\quad (1-4t(\tau))(1-\l(2\tau))^{\frac12}\theta_3(2\tau)^6=(1-\l(4\tau))\theta_3(4\tau)^6.
    \]

    For a given $m$, we fix the choices $\l^{m/24}$ and $(1-\l)^{m/24}$. Similarly, we take $(1-\l(2\tau))^{1/2}=1-8t(\tau)$. Then 
    \begin{equation*}
        \begin{split}
            2\cdot 2^{\frac{mj}{12}}&\BK_2\left(mj/24,1-mj/24\right)\left(48\tau/m\right) \\&= 4\cdot  2^{4\cdot \frac{mj}{24}} \left(\frac\l{(1-\l)^2}\right)^{mj/24}(1-\l)\theta_3^6\left(48\tau/m\right)\\
            &=4\left(\frac{\l^2}{1-\l}\right)^{mj/24}(1-\l)^{\frac 12} \theta_3^6\left(24\tau/m\right) (1-4t(12\tau/m))\\
            &= 2 \left(\frac{\l^2}{1-\l}\right)^{mj/24}(1-\l)^{\frac 12} (1+(1-\l)^{\frac12})\theta_3^6\left(24\tau/m\right)\\
            &=2\left(\frac{\l^2}{(1-\l)^2}\right)^{mj/24}(1-\l)^{\frac{mj}{24}} (1-\l)^{\frac12}(1+(1-\l)^{\frac12})\theta_3^6\left(24\tau/m\right)\\
            &= 2^{\frac{mj}{12}}\BK_2\left(\frac{mj}{12},\frac 12+\frac{mj}{24}\right) \left(24\tau/m\right) + 2^{\frac{mj}{12}}\BK_2\left(\frac{mj}{12},1+\frac{mj}{24}\right) \left(24\tau/m\right).
        \end{split}
    \end{equation*}
    By the same argument, we have
    \begin{multline*}
        8\BK_2\left(\frac 12+\frac{mj}{24},\frac32-\frac{mj}{24}\right)\left(\frac{48}m\tau\right) =\\ 
        \BK_2\left(\frac{mj}{12},\frac  12+\frac{mj}{24}\right) \left(\frac{24}m\tau\right)  - \BK_2\left(\frac{mj}{12},1+\frac{mj}{24}\right) \left(\frac{24}m\tau\right).
    \end{multline*}
    These equations are equivalent to the identities stated in the proposition and the corresponding $P$-values  follow from \Cref{lem:K-Lvalue}. We will omit the details. 
\end{proof}
\subsection{Special \texorpdfstring{$L$}{L}-values  of \texorpdfstring{$f_{3,D}^\sharp$}{f3D} and some of their twists} 
Using identities in the previous subsection, one obtains the following applications.  For a given Hecke eigenform $f^\sharp$, the relationships listed below induce the functional equations  relating $L(f^\sharp,2)$ and $L(f^\sharp\otimes \chi,1)$ for an odd inner twist $\chi$ modulo $N$, which is guaranteed by the result of Shimura \cite[Theorem 1]{Shimura-special-zeta} that 
\[
    L(f^\sharp,2) \in \g(\chi) \pi i  L(f^\sharp\otimes \chi, 1) \cdot \Q(a_n(f^\sharp): n\geq 1),
\]
if $\chi(-1)=-1$,  where $\g(\chi)=\sum_{j=0}^{N-1}e^{2\pi i j/N}\chi(j)$.  In the end of this section, we provide the  $L$-values of the chosen newforms $f^\sharp$ and some functional equations relating $L(f^\sharp,2)$ and $L(f^\sharp\otimes \chi,1)$.
\begin{corollary}
    For $m=8$, let $f^\sharp$ be the newform $f_{36.3.d.a}(\tau)=\BK_{2} \left(1/3,2/3 \right)(6\tau)-2\BK_{2} \left(2/3,4/3 \right)(6\tau)$. Then 
    \[
        (1+\sqrt 3) P\left(2/3,4/3\right)=\frac{B(1/4,1/4)^2}{  2^{\frac56} 3^{\frac14}} =6\pi\cdot  2^{2/3}\cdot L(f^\sharp, 1)= 18\cdot 2^{2/3}\cdot  L(f^\sharp, 2).
    \]
\end{corollary} 
This result follows from \Cref{lem:K-Lvalue} and \Cref{prop: ALids-m=8}.
\begin{corollary}
    The transcendental part of the Petersson inner product of $f_{36.3.d.a}$ or $f_{36.3.d.b}$ is $B(1/4,1/4)^4/\pi^5$. 
\end{corollary}
This result follows from the work of  Shimura \cite[Theorem 4]{Shimura-special-zeta},  \Cref{prop: ALids-m=8} and the fact  $f_{36.3.d.b}=f_{36.3.d.a}\otimes \chi_{-3}$.
\begin{corollary}
    Let $m=3$ and $f^\sharp$ be the newform given in \Cref{tab:G2-eigenforms}.  Then 
    \[
       L(f^\sharp, 1) = \frac1{16\pi} \left((2+\sqrt 2)(1+i)P(1/4,5/8)+ \sqrt 2(1-i) P(1/4,9/8)\right),
    \]
    \begin{equation*}
        \begin{split} 
            L(f^\sharp, 2) &=\frac 1{64} \left((2+\sqrt 2)P(1/4,5/8)+\sqrt{2}iP(1/4,9/8)\right) \\
            &= \frac{1+i}8\pi \cdot L(f^\sharp\otimes\chi_{-2},1) =\frac{(i-1)(1+\sqrt2)}8\pi \cdot L(f^\sharp\otimes\chi_{-1},1).
        \end{split}
    \end{equation*}
\end{corollary}
\begin{proof}
    According to  \Cref{cor:AAR3.3.5} and \Cref{prop: ALids}, we have 
    \[
        P(3/4,7/8)=(2+\sqrt 2)P(1/4,5/8), \quad P(3/4,11/8)=(2-\sqrt 2)P(1/4,9/8),
    \]
    \begin{equation*}
        \begin{split} 
            \sqrt{2} P(1/8,7/8)&=P(1/4,5/8)+P(1/4,9/8), \\
            \sqrt{2} P(5/8,11/8)&=P(1/4,5/8)-P(1/4,9/8),\\
            2 P(3/8,5/8)&=(\sqrt 2+1)P(1/4,5/8)+(\sqrt 2-1)P(1/4,9/8),\\
            2 P(7/8,9/8)&=(\sqrt 2+1)P(1/4,5/8)-(\sqrt 2-1)P(1/4,9/8).
        \end{split}
    \end{equation*}
    Hence by the integral formula (\Cref{lem:K-Lvalue}) and above identities, 
    \begin{equation*}
        \begin{split} 
            {16\pi}L(f^\sharp, 1) &= \sqrt 2\left(P(1/8,7/8)+\sqrt 2 P(3/8,5/8)+i P(5/8,11/8)+\sqrt 2iP(7/8,9/8) \right)\\
            &= (2+\sqrt 2)(1+i)P(1/4,5/8)+ \sqrt 2(1-i) P(1/4,9/8), 
        \end{split}
    \end{equation*}
    and
    \begin{equation*}
        \begin{split} 
            16^2L(f^\sharp, 2) &=2^{3/2} P(3/4,7/8)+4P(1/4,5/8)+2\sqrt 2i P(3/4,11/8)+4iP(1/4,9/8) \\
           &= 4(2+\sqrt 2)P(1/4,5/8)+4\sqrt{2}iP(1/4,9/8). 
        \end{split}
    \end{equation*}
    On the other hand, one can simplify $L(f^\sharp\otimes \chi_{-1}, 1)$ and $L(f^\sharp\otimes \chi_{-2}, 1)$ as follows.
    \begin{equation*}
        \begin{split} 
            \frac{16\pi}{\sqrt 2} L(f^\sharp\otimes \chi_{-1}, 1)
            &=P(1/8,7/8)+iP(5/8,11/8)-\sqrt 2 P(3/8,5/8)-\sqrt 2iP(7/8,9/8)\\
            &= -(1+i)P(1/4,5/8)+(\sqrt{2}-1)(1-i)P(1/4,9/8)\\
            &=\frac{(1-\sqrt 2)(1+i)}{\sqrt 2} \left((2+\sqrt 2)P(1/4,5/8)+\sqrt{2}iP(1/4,9/8\right)\\
            &=\frac{(1-\sqrt 2)(1+i)}{\sqrt 2}  64\cdot L(f^\sharp, 2);
        \end{split}
    \end{equation*}
    \begin{equation*}
        \begin{split} 
            \frac{16\pi}{\sqrt 2} L(f^\sharp\otimes \chi_{-2}, 1) &=P(1/8,7/8)-iP(5/8,11/8)+\sqrt 2 P(3/8,5/8)-\sqrt 2iP(7/8,9/8)\\
            &= (\sqrt 2+1)(1-i)P(1/4,5/8)+i(1-i)P(1/4,9/8)\\
            &=\frac{(1-i)}{\sqrt 2}  64\cdot L(f^\sharp, 2),
        \end{split}
    \end{equation*}
    which give the identities in the corollary.
\end{proof}
Other functional equations listed in the table can be deduced in a similar manner; for interested readers, please see \cite{ENRosen}.  
%
%
\begin{table}[ht]
    \centering
    \begin{tblr}{
        colspec={|Q[c,m]|X[l,20]|},
        hlines,
        vlines,
        rowsep={4pt},
        hline{1,2,Z}={1pt},
        vline{1,2,Z}={1pt},
    }
    $m$ & $L(f^\sharp,1)$ \\
    $1$  & {$\displaystyle \pi L(f_{1152.3.b.i}(\tau),1) =$ \\ $ \displaystyle \frac{1}{192 \cdot 2^{\frac 16}} \bigg[ 8 P\left(1/24\right) - 4 \beta_{5}P \left(7/24\right) -2 \beta_{2} P\left(13/24\right) - \beta_{7} P\left(19/24 \right)\bigg]$ \\
    $\displaystyle -\frac{1}{192 \cdot 2^{\frac{5}{6}} } \bigg[8 \beta_{1}P\left(5/24\right) + 4 \beta_{3} P\left(11/24\right) + 2 \beta_{4}P \left(17/24\right) + \beta_{6} P\left(23/24\right) \bigg]$} \\
    $2$   &   {$\pi L(f_{288.3.g.a}(\tau),1)=\displaystyle \frac{1}{12 \cdot 2^{\frac{1}{3}}} \left[P \left(1/12\right) + i \cdot P \left(7/12\right) \right] - \frac{1}{6 \cdot 2^{\frac{2}{3}}} \left[P \left(5/12\right) + i \cdot P \left(11/12\right) \right]$}  \\ 
    $3$ & {$\pi L(f_{128.3.d.c}(\tau),1) = \displaystyle\frac{1}{8 \sqrt{2}} \left[ P\left(1/8 \right) + i \cdot P \left(5/8\right) \right]+\frac{1}{8} \left[P \left(3/8\right) + i \cdot P \left(7/8\right) \right]$}\\ 
    $6$ &  $\pi L(f_{32.3.c.a}(\tau),1) = \frac{1}{8} \left( \, P\left(1/4\right) +i \cdot \, P\left(3/4\right)\right)$ \\
    4, 8 &  $\displaystyle \pi L(f_{36.3.d.a},1)= \frac{\pi}{\sqrt{3}} L(f_{36.3.d.b}, 1) =\frac{B(1/4,1/4)^2}{  2^{\frac52} 3^{\frac54}}$ \\
    \end{tblr}
    \caption{G2 $L$-values in terms of $P(m/24) = P\left(m/24, \frac{24-m}{24} + \lfloor \frac{m}{12}\rfloor\right)$ \\ In the above, {\footnotesize $
    \beta_{1}=2\sqrt{-7},\, \beta_2=4\sqrt{-1},\, \beta_3=4\sqrt2,\, \beta_4=8\sqrt 7, \, \beta_{5}=2\sqrt{-14},\, \beta_6=-16\sqrt{-2},\, \beta_7=8\sqrt{14}.
    $}}
    \label{tab:G2-L-values}
\end{table}

\begin{table}[ht]
    \centering
    \begin{tblr}{
        colspec={|Q[c,m]|Q[l,m]|},
        hlines,
        vlines,
        rowsep={3pt},
        hline{1,2,Z}={1pt},
        vline{1,2,Z}={1pt},
    }
    $m$& \SetCell{c} $L(f^\sharp,2)$ \\
    $1$ &$\displaystyle L(f^\sharp,2)=-\frac{1+i}{12}\pi \cdot L(f^\sharp\otimes\chi_{-2},1)$ \\
    $2$ &$\displaystyle L(f_{288.3.g.a},2)=-\frac{1+i}{12}\pi \cdot L(f_{288.3.g.a}\otimes\chi_{-1},1)=\frac1{\sqrt{3}}\frac{1-i}{12}\pi \cdot L(f_{288.3.g.c},1)$\\
    $3$ &$\displaystyle L(f^\sharp,2)=\frac{1+i}8\pi \cdot L(f^\sharp\otimes\chi_{-2},1) =\frac{(i-1)(1+\sqrt2)}8\pi \cdot L(f^\sharp\otimes\chi_{-1},1)$\\
    $6$ &$\displaystyle L(f^\sharp,2)=\frac{1+i}4\pi \cdot L(f^\sharp\otimes\chi_{-1},1) $\\
    $4, 8$ & $ L(f^\sharp,2)=\frac{\pi}3\cdot L(f^\sharp,1)$\\
    \end{tblr}
    \captionsetup{justification=centering}
    \caption{G2 $L$-values at 2.}
    \label{tab:G2-L-values-2}
\end{table}
\newpage
\bibliographystyle{plain}
\bibliography{ref}

\begin{thebibliography}{10}

\bibitem{HMM1}
Michael Allen, Brian Grove, Ling Long, and Fang-Ting Tu.
\newblock The {E}xplicit {H}ypergeometric-{M}odularity {M}ethod {I}, 2024
  arXiv: 2404.00711.

\bibitem{AAR}
George~E. Andrews, Richard Askey, and Ranjan Roy.
\newblock {\em Special functions}, volume~71 of {\em Encyclopedia of
  Mathematics and its Applications}.
\newblock Cambridge University Press, Cambridge, 1999.

\bibitem{Berndt-Evans-Williams}
Bruce~C. Berndt, Ronald~J. Evans, and Kenneth~S. Williams.
\newblock {\em Gauss and {J}acobi sums}.
\newblock Canadian Mathematical Society Series of Monographs and Advanced
  Texts. John Wiley \& Sons, Inc., New York, 1998.
\newblock A Wiley-Interscience Publication.

\bibitem{BCM}
Frits Beukers, Henri Cohen, and Anton Mellit.
\newblock Finite hypergeometric functions.
\newblock {\em Pure Appl. Math. Q.}, 11(4):559--589, 2015.

\bibitem{Beukers-Frederic}
Frits Beukers and Fr\'ed\'eric Jouhet.
\newblock Duality relations for hypergeometric series.
\newblock {\em Bull. Lond. Math. Soc.}, 47(2):343--358, 2015.

\bibitem{BB}
Jonathan~M. Borwein and Peter~B. Borwein.
\newblock {\em Pi and the {AGM}}, volume~4 of {\em Canadian Mathematical
  Society Series of Monographs and Advanced Texts}.
\newblock John Wiley \& Sons, Inc., New York, 1998.
\newblock A study in analytic number theory and computational complexity,
  Reprint of the 1987 original, A Wiley-Interscience Publication.

\bibitem{CohenII}
Henri Cohen.
\newblock {\em Number theory. {V}ol. {II}. {A}nalytic and modern tools}, volume
  240 of {\em Graduate Texts in Mathematics}.
\newblock Springer, New York, 2007.

\bibitem{Cohen-Compute-L}
Henri Cohen.
\newblock Computing {$L$}-functions: a survey.
\newblock {\em J. Th\'eor. Nombres Bordeaux}, 27(3):699--726, 2015.

\bibitem{WIN3a}
Alyson Deines, Jenny~G. Fuselier, Ling Long, Holly Swisher, and Fang-Ting Tu.
\newblock Generalized {L}egendre curves and quaternionic multiplication.
\newblock {\em J. Number Theory}, 161:175--203, 2016.

\bibitem{DPVZ}
Lassina Dembélé, Alexei Panchishkin, John Voight, and Wadim Zudilin.
\newblock {S}pecial {H}ypergeometric {M}otives and {T}heir {L}-functions: Asai
  recognition.
\newblock {\em Experimental Mathematics}, 0(0):1--23, 2020.

\bibitem{Dwork}
Bernard Dwork.
\newblock {$p$}-adic cycles.
\newblock {\em Inst. Hautes \'Etudes Sci. Publ. Math.}, 37:27--115, 1969.

\bibitem{Win3X}
Jenny Fuselier, Ling Long, Ravi Ramakrishna, Holly Swisher, and Fang-Ting Tu.
\newblock Hypergeometric functions over finite fields.
\newblock {\em Mem. Amer. Math. Soc.}, 280(1382), 2022.

\bibitem{Greene}
John Greene.
\newblock Hypergeometric functions over finite fields.
\newblock {\em Trans. Amer. Math. Soc.}, 301(1):77--101, 1987.

\bibitem{DMgrove}
Brian Grove.
\newblock On the {H}ypergeometric {M}odularity {C}onjectures of {D}awsey and
  {M}c{C}arthy, 2024 (preprint).

\bibitem{KatzESDE}
Nicholas~M. Katz.
\newblock {\em Exponential sums and differential equations}, volume 124 of {\em
  Annals of Mathematics Studies}.
\newblock Princeton University Press, Princeton, NJ, 1990.

\bibitem{Katz09}
Nicholas~M. Katz.
\newblock Another look at the {D}work family.
\newblock In {\em Algebra, arithmetic, and geometry: in honor of {Y}u. {I}.
  {M}anin. {V}ol. {II}}, volume 270 of {\em Progr. Math.}, pages 89--126.
  Birkh\"{a}user Boston, Boston, MA, 2009.

\bibitem{LLT2}
Wen-Ching~Winnie Li, Ling Long, and Fang-Ting Tu.
\newblock A {W}hipple {$_7F_6$} formula revisited.
\newblock {\em La Matematica}, 1(2):480--530, 2022.

\bibitem{LMFDB}
The {LMFDB Collaboration}.
\newblock The {L}-functions and modular forms database.
\newblock \url{http://www.lmfdb.org}, 2021.
\newblock [Online; accessed 25 August 2021].

\bibitem{Long}
Ling Long.
\newblock Hypergeometric evaluation identities and supercongruences.
\newblock {\em Pacific J. Math.}, 249(2):405--418, 2011.

\bibitem{LR}
Ling Long and Ravi Ramakrishna.
\newblock Some supercongruences occurring in truncated hypergeometric series.
\newblock {\em Adv. Math.}, 290:773--808, 2016.

\bibitem{Long-Yang}
Ling Long and Yifan Yang.
\newblock Hodge numbers of hypergeometric data.
\newblock {\em arXiv:2404.02834}, 2024.

\bibitem{McCarthy}
Dermot McCarthy.
\newblock Transformations of well-poised hypergeometric functions over finite
  fields.
\newblock {\em Finite Fields Appl.}, 18(6):1133--1147, 2012.

\bibitem{McCarthy-Papanikolas}
Dermot McCarthy and Matthew~A. Papanikolas.
\newblock A finite field hypergeometric function associated to eigenvalues of a
  {S}iegel eigenform.
\newblock {\em Int. J. Number Theory}, 11(8):2431--2450, 2015.

\bibitem{Mor-gamma}
Yasuo Morita.
\newblock A {$p$}-adic analogue of the {$\Gamma $}-function.
\newblock {\em J. Fac. Sci. Univ. Tokyo Sect. IA Math.}, 22(2):255--266, 1975.

\bibitem{Osburn-Straub}
Robert Osburn and Armin Straub.
\newblock Interpolated sequences and critical {$L$}-values of modular forms.
\newblock In {\em Elliptic integrals, elliptic functions and modular forms in
  quantum field theory}, Texts Monogr. Symbol. Comput., pages 327--349.
  Springer, Cham, 2019.

\bibitem{RRV22}
David~P. Roberts and Fernando Rodriguez~Villegas.
\newblock Hypergeometric motives.
\newblock {\em Notices Amer. Math. Soc.}, 69(6):914--929, 2022.

\bibitem{Rogers-Zudilin}
Mathew Rogers and Wadim Zudilin.
\newblock From {$L$}-series of elliptic curves to {M}ahler measures.
\newblock {\em Compos. Math.}, 148(2):385--414, 2012.

\bibitem{Rosen-K1}
Esme Rosen.
\newblock Modular forms and certain ${}_2{F}_1(1)$ hypergeometric series, (in
  preparation) 2024.

\bibitem{ENRosen}
Esme Rosen.
\newblock Transcendence of $_3{F}_2(1)$ hypergeometric series and
  \textit{L}-values of modular forms, (preprint) 2024.

\bibitem{Shimura-special-zeta}
Goro Shimura.
\newblock The special values of the zeta functions associated with cusp forms.
\newblock {\em Comm. Pure Appl. Math.}, 29(6):783--804, 1976.

\bibitem{Silverman2}
Joseph~H. Silverman.
\newblock {\em Advanced topics in the arithmetic of elliptic curves}, volume
  151 of {\em Graduate Texts in Mathematics}.
\newblock Springer-Verlag, New York, 1994.

\bibitem{Waldschmidt}
Michel Waldschmidt.
\newblock Elliptic functions and transcendence.
\newblock In {\em Surveys in number theory}, volume~17 of {\em Dev. Math.},
  pages 143--188. Springer, New York, 2008.

\bibitem{Watkins-HGM-documentation}
Mark Watkins.
\newblock Hypergeometric motives over {$\mathbb Q$} and their {$L$}-functions.
\newblock In {\em MAGMA Documentation}.
  \url{http://magma.maths.usyd.edu.au/~watkins/papers/known.pdf}, 2017.
\newblock preprint.

\bibitem{Weil52}
Andr\'{e} Weil.
\newblock Jacobi sums as ``{G}r\"{o}ssencharaktere''.
\newblock {\em Trans. Amer. Math. Soc.}, 73:487--495, 1952.

\bibitem{Whipple25}
F.~J.~W. Whipple.
\newblock On {W}ell-{P}oised {S}eries, {G}eneralized {H}ypergeometric {S}eries
  having {P}arameters in {P}airs, each {P}air with the {S}ame {S}um.
\newblock {\em Proc. London Math. Soc. (2)}, 24(4):247--263, 1925.

\bibitem{Zagier-modularform}
Don Zagier.
\newblock Elliptic modular forms and their applications.
\newblock In {\em The 1-2-3 of modular forms}, Universitext, pages 1--103.
  Springer, Berlin, 2008.

\bibitem{Zudilin-well-poised-Euler}
Wadim Zudilin.
\newblock Well-poised hypergeometric transformations of {E}uler-type multiple
  integrals.
\newblock {\em J. London Math. Soc. (2)}, 70(1):215--230, 2004.

\end{thebibliography}

\address{Department of Mathematics, Louisiana State University, Baton Rouge, LA 70803, USA}

\email{\href{mailto:allenm3@lsu.edu}{allenm3@lsu.edu}}, {\url{https://michaelgallen.com/}}

\email{\href{mailto:bgrove2@lsu.edu}{bgrove2@lsu.edu}}, {\url{https://briandgrove.wordpress.com/}}

\email{\href{mailto:llong@lsu.edu}{llong@lsu.edu}}, {\url{https://www.math.lsu.edu/~llong/}}

\email{\href{mailto:ftu@lsu.edu}{ftu@lsu.edu}}, {\url{https://sites.google.com/view/ft-tu/}}

\end{document}